\documentclass[11pt,a4]{amsart}
\usepackage{amsfonts,amsmath,amssymb}
\usepackage{mathabx}
\usepackage{mathrsfs}

\usepackage[latin1]{inputenc}
\usepackage[usenames,dvipsnames]{pstricks}
\newtheorem{theorem}{Theorem}[section]

\newtheorem{Corol}[theorem]{Corollary}

\newtheorem{Lemm}[theorem]{Lemma}

\newcommand{\p}{\partial}
\newcommand{\R}{\mathbb{R}}
\newcommand{\C}{\mathbb{C}}

\newcommand{\abs}[1]{\left\vert#1\right\vert}
\newcommand{\set}[1]{\left\{#1\right\}}
\newcommand{\para}[1]{\left(#1\right)}

\usepackage{graphicx}

\usepackage{mathrsfs}
\usepackage[latin1]{inputenc}
\usepackage[T1]{fontenc}
\usepackage{color,xspace}
\usepackage{hyperref}
\usepackage[usenames,dvipsnames]{pstricks}
\usepackage{times}
\allowdisplaybreaks
\begin{document}
\title[Wave equation with time-dependent coefficients]{Stable determination outside a cloaking region of two time-dependent coefficients in an hyperbolic equation from Dirichlet to Neumann map}
\author[M.~Bellassoued]{Mourad~Bellassoued}
\author[I.~Ben A\"{\i}cha]{Ibtissem Ben A\"{\i}cha}
\address{M.~Bellassoued.  University of Tunis El Manar, National Engineering School of Tunis, ENIT-LAMSIN, B.P. 37, 1002 Tunis, Tunisia}
\email{mourad.bellassoued@fsb.rnu.tn}
\address{I.~Ben A\"{\i}cha. University of Aix-Marseille, 58 boulevard Charles Livon, 13284 Marseille,  France. \& University of Carthage,
Faculty of Sciences of Bizerte, 7021 Jarzouna Bizerte, Tunisia.}
\email{ibtissembenaicha91@gmail.com} \maketitle
\begin{abstract}
In this paper, we treat the inverse problem of determining two
time-dependent coefficients appearing in a dissipative wave equation,
from measured Neumann boundary observations. We establish in dimension $n\geq2$, stability estimates with respect to the
Dirichlet-to-Neumann map of these coefficients  provided that are known outside a cloaking regions. Moreover, we prove that it can be stably recovered in larger subsets of the domain by enlarging
the set of
data.
\end{abstract}
\section{Introduction}
\subsection{Statement of the problem}
This paper deals with the inverse problem of determining two time-dependent coefficients in a dissipative wave equation from
boundary observations. Let $\Omega$ be a bounded domain of $\R^{n}$, $n\geq
2$, with $\mathcal{C}^{\infty}$ boundary $\Gamma=\p \Omega$. Given $T> 0$, we
introduce the following dissipative wave equation
\begin{equation}\label{EQ1} \left\{
  \begin{array}{ll}
    \p_{t}^{2}u-\Delta u +a(x,t)\p_{t}u+b(x,t)u=0\, & \mbox{in}\,\,\,Q=\Omega\times (0,T), \\
\\
   u(x,0)=u_{0}(x),\,\,\,\,\p_{t}u(x,0)=u_{1}(x)  & \mbox{in}\,\,\,\Omega, \\
\\
   u(x,t)=f(x,t)  & \mbox{on}\,\,\,\Sigma=\Gamma\times (0,T),
  \end{array}
\right.
\end{equation}
where $f\in H^{1}(\Sigma)$, $u_{0}\in H^{1}(\Omega)$, $u_{1}\in
L^{2}(\Omega)$, and the coefficients  $a\in\mathcal{C}^{2}({Q})$ and
$b\in\mathcal{C}^{1}({Q})$ are assumed to be real valued. It is well known
(see \cite{[R15]}) that if $f(\cdot,0)=u_{0|\Gamma}$, there exists a unique solution
$u$ to the equation (\ref{EQ1}) satisfying
$$u\in \mathcal{C}([0,T],H^{1}(\Omega))\cap \mathcal{C}^{1}([0,T],L^{2}(\Omega)).$$
Moreover, there exists $C>0$, such that
\begin{equation}\label{estim energie}
\|\p_{\nu}u\|_{L^{2}(\Sigma)}+\|u(\cdot,t)\|_{H^{1}(\Omega)}+\|\p_{t}u(\cdot,t)\|_{L^{2}(\Omega)}\leq C \para{\|f\|_{H^{1}(\Sigma)}+\|u_{0}\|_{H^{1}(\Omega)}
+\|u_{1}\|_{L^{2}(\Omega)}}.
\end{equation}
Here $\nu$ denotes the unit outward normal to $\Gamma$ at $x$ and $\p_{\nu}u$
stands for $\nabla u\cdot\nu$.

 In the present paper, we address the
uniqueness and the stability issues in the study of an inverse problem for
the dissipative wave equation (\ref{EQ1}), in the presence of an absorbing
coefficient $a$ and a potential $b$ that depend on both space and time
variables.  We consider three different sets of data and we aim to show that
$a$ and $b$ can be recovered in some specific subsets of the domain, by
probing it with disturbances generated on the boundary.  The Dirichlet data
$f$ is considered as a disturbance that is used to probe the medium which is
assumed to be quiet initially.
\medskip

The problem of identifying coefficients appearing in hyperbolic boundary
value problems was treated very well and there are many works that are
relevant to this topic.  In the case where the unknown coefficient is
depending only on the spatial variable, Rakesh and Symes \cite{[R23]} proved
by means of geometric optics solutions, a uniqueness result in recovering a
time-independent potential in a wave equation from global Neumann data. The
uniqueness by local Neumann data, was considered by Eskin \cite{[R13]} and
Isakov \cite{[R17]}.  In \cite{[R5]}, Bellassoued, Choulli and Yamamoto
proved a log-type stability estimate, in the case where the Neumann data are
observed on any arbitrary subset of the boundary.  Isakov and Sun
\cite{[R19]} proved that the knowledge of local Dirichlet-to-Neumann map
yields a stability result of H\"older type in determining a coefficient in a
subdomain. As for the stability obtained from global Neumann data, one can
see Sun \cite{[R30]}, Cipolatti and Lopez \cite{[R11]}. The case of
Riemannian manifold was considered  by  Bellassoued and Dos  Santos
Ferreira \cite{[R6]}, Stefanov and Uhlmann
 \cite{[R29]}.
 \medskip

All the mentioned papers are concerned only with time-independent coefficients.
 In the case where the coefficient is also depending on
the time variable, There is a uniqueness result proved by  Ramm and Rakesh
\cite{[R24]}, in which they showed that a time-dependent coefficient
appearing in a wave equation with zero initial conditions, can be uniquely
determined from the knowledge of global Neumann data, but only in a precise
subset of the cylindrical domain $Q$ that is  made of lines making an angle
of $45^{\circ}$ with the $t$-axis and meeting the planes $t=0$ and $t=T$
outside $\overline{Q}$.  However, inspired by the work of \cite{[R31]},
Isakov proved in \cite{[R18]}, that the time-dependent coefficient can be recovered from the responses of the medium for all
possible initial data,
over the whole domain $Q$.
\medskip

It is clear that with zero initial data, there is no hope to recover a
time-dependent coefficient appearing in a hyperbolic equation over the whole
cylindrical  domain, even from the knowledge of global Neumann data, because
 the value of the solution can be effected by the value of the
initial conditions, which is actually due  to a fundamental concept concerning
hyperbolic equations called the domain of dependence (see \cite{[R15]}). Moreover, we can prove that the  backward light-cone with base $\Omega$ is a cloaking region, that is we can not uniquely recover the coefficients in this region.
\medskip

As for uniqueness results, we have also the paper of Stefanov \cite{[R27]},
in which he proved that a time-dependent potential appearing in a wave
equation can be uniquely recovered  from scattering data and the paper of
Ramm and Sj\"{o}strand  \cite{[R25]}, in which they proved a uniqueness
result on an infinite time-space cylindrical domain $\Omega\times\R_{t}$.
\medskip

 The stability in this case, was considered by  Salazar \cite{[R26]}, who extended  the result of Ramm and Sj\"{o}strand \cite{[R25]}  to
more general coefficients and he established a stability result for compactly
supported coefficients provided $T$ is sufficiently large. We also refer to
the works of Kian \cite{[R20],[R21]} who followed techniques used by
Bellassoued,  Jellali and  Yamamoto \cite{[R7],[R8]}  and proved
uniqueness and a $\log$-$\log$ type stability estimate from the knowledge of
partial Neumann data. As for stability results from global Neumann data we
refer to  Ben A\"icha \cite{[R10]} who proved recently a stability of
$\log$-type in recovering a zeroth order time-dependent coefficient in
different regions of the cylindrical domain by considering different sets of
data. We  also refer to Waters \cite{[R33]} who derived, in Riemannian
case, conditional H\"older stability estimates for the X-ray transform of the
time-dependent potential appearing in the wave equation.
\medskip

As for results of hyperbolic inverse problems dealing with single measurement
data, one can see \cite{[R2],[R3],[R9],[R12],[R16],[R28]} and the references
therein.
\medskip

Inspired by the work of  Bellassoued \cite{[R4]} and  following the same
strategy as in  Ben A\"icha \cite{[R10]}, we  prove in this paper stability
estimates in the recovery of the unknown coefficients $a$ and $b$ via
different types of measurements and over different subsets of the domain $Q$.
\subsection{ Main results }
Before stating our main results we first introduce the following notations.

Let $r>0$ be such that $T>2r$ and $\Omega\subseteq B(0,r/2)=\Big\{x\in
\R^{n},\,|x|<r/2\Big\}$. We set $Q_{r}=B(0,r/2)\times(0,T)$ and we consider the annular
region around the domain $\Omega$,
$$\mathscr{A}_{r}=\left\{x\in\R ^{n},\,\,\displaystyle\frac{r}{2}<|x|<T-\displaystyle\frac{r}{2}\right\},$$
and the forward and backward light cone:
$$
\mathscr{C}_{r}^{+}=\left\{(x,t)\in Q_{r},\,\,|x|<t-\displaystyle\frac{r}{2},\,t>\frac{r}{2}\right\},$$
$$\mathscr{C}_{r}^{-}=\left\{(x,t)\in Q_{r},
\,\,|x|<T-\displaystyle\frac{r}{2}-t,\,T-\frac{r}{2}> t\right\}.$$ Finally,
we denote 
$$ Q_{r}^{*}=\mathscr{C}_{r}^{+}\cap \mathscr{C}_{r}^{-}
\qquad Q_{r,*}= Q\cap
Q_{r}^{*},\qquad \textrm{and}\,quad Q_{r,\sharp}=Q\cap
\mathscr{C}^{+}_{r}.$$

 We remark that the open subset $Q_{r,*}$ is made of
lines making an angle of $45^{\circ}$
 with the $t$-axis and meeting the planes $t=0$ and $t=T$ outside
$\overline{Q}_{r}$ and $Q_{r,\sharp}$ is made of lines making an angle of
$45°$ with the $t$-axis and meeting only the planes $t=0$ outside
$\overline{Q}_{r}$. We notice that $Q_{r,*}\subset Q_{r,\sharp}\subset Q$.
Note, that in the particular case where $\Omega=B(0,r/2),$ we have
$Q_{r,*}=Q_{r}^{*}$ and $Q_{r,\sharp}=\mathscr{C}^{+}_{r}$ (see Figure 1 in
\cite{[R10]}).

 In the present paper, we will prove stability estimates
for the inverse problem under consideration in three cases. We will consider
three different sets of data and we will prove that the coefficients $a$ and
$b$ can be stably determined in three different regions of the cylindrical
domain $Q$. Given $M_{1},\,M_{2}>0,$  we consider the set of admissible
coefficients $a$ and $b$:
$$\mathcal{A}(M_{1},M_{2})=\{(a,b)\in\mathcal{C}^{2}(\overline{Q}_{r})\times \mathcal{C}^{1}(\overline{Q}_{r});\,\,
\|a\|_{C^{2}(Q)}\leq M_{1},
\,\,\|b\|_{C^{1}(Q)} \leq M_{2}\}.$$
Finally, we define the following space
$$\mathcal{H}_{0}^{1}=\Big\{ f\in H^{1}(\Sigma),\,f(\cdot,0)=0\,\,\mbox{on}\,\,\Gamma \Big\},$$
equipped with the norm of $H^{1}(\Sigma).$
\subsubsection{Determination of coefficients from boundary measurements:}
 In the first case, we will assume that the initial conditions $(u_{0},u_{1})$ are zero and our set of data will be given only
 by boundary measurements \textit{enclosed} by the
  Dirichlet-to-Neumann map $\Lambda_{a,b}$ defined as follows
$$
\begin{array}{rrr}
 \Lambda_{a,b} : &\mathcal{H}_{0}^{1}(\Sigma)\longrightarrow& L^{2}(\Sigma)\\
           &f\longmapsto& \p_{\nu}u.
\end{array}
$$
 Note that from (\ref{estim energie}) we have $\Lambda_{a,b}$ is continuous from $\mathcal{H}_{0}^{1}(\Sigma)$
to $L^{2}(\Sigma)$. We denote by $\|\Lambda_{a,b}\|$ its norm in
$\mathcal{L}(\mathcal{H}_{0}^{1}(\Sigma),L^{2}(\Sigma))$. We can prove that it is hopeless to uniquely determine $a$ and $b$ in the case where these coefficients are supported in the cloaking region 
$$
\mathscr{C}=\set{(x,t)\in Q_r;\,\, \abs{x}<\frac{r}{2}-t,\,\, t<\frac{r}{2}}.
$$
The first result of this paper can be stated as follows
\begin{theorem}\label{Theorem1}
Let $T\!> \!2\!$ Diam$(\Omega)$. There exist constants $C>0$ and
$\mu_{1},\,\mu_{2}\in (0,1)$ such that we have
$$\|a_{2}-a_{1}\|_{L^{\infty}(Q_{r,*})}\leq
 C \Big( \|\Lambda_{a_{1},b_{1}}-\Lambda_{a_{2},b_{2}}\|^{\mu_{1}}+|\log \|\Lambda_{a_{1},b_{1}}-\Lambda_{a_{2},b_{2}}\||^{-1}
 \Big)^{\mu_{2}},$$
for any $(a_{i},b_{i})\in\mathcal{A}(M_{1},M_{2})$ such that
$\|a_{i}\|_{H^{p}(Q)}\leq M_{1} $, for  some $p>n/ 2 +3/2$,
$(a_{1},b_{1})=(a_{2},b_{2})$ in $\overline{Q}_{r}\setminus Q_{r,*}$ and
$(\nabla a_{1},\nabla b_{1})=(\nabla a_{2},\nabla b_{2})$ on $\p Q_{r}\cap\p
Q_{r,*}$. Here $C$ depends only on $\Omega$, $M_{1}$, $M_{2}$, $T$ and $n$.
\end{theorem}
The above statement claims stable determination of the time-dependent
coefficient $a$ from the Neumann boundary measurements $\Lambda_{a,b}$ in
$Q_{r,*}\!\subset\! Q$, provided $a$ is known outside $Q_{r,*}$. By Theorem
\ref{Theorem1}, we can readily derive the following result
\begin{theorem}\label{Theorem2}
Let $T\!> \!2\!$ Diam$(\Omega)$. There exist two constants $C>0$ and $\mu\in
(0,1)$ such that we have
$$\|b_{2}-b_{1}\|_{H^{-1}(Q_{r,*})}\leq C \Big( \log |\log \|\Lambda_{a_{2},b_{2}}-\Lambda_{a_{1},b_{1}}\||^{\mu} \Big)^{-1}, $$
for any $(a_{i},b_{i})\in \mathcal{A}(M_{1},M_{2})$ such that
$\|a_{i}\|_{H^{p}(Q)}\leq M_{1}$, for some $p>n/2+3/2$,
$(a_{1},b_{1})=(a_{2},b_{2})$ in $\overline{Q}_{r}\setminus Q_{r,*}$ and
$(\nabla a_{1},\nabla b_{1})=(\nabla a_{2},\nabla b_{2})$ on $\p Q_{r}\cap\p
Q_{r,*}$. Here $C$ depends only on $\Omega$, $M_{1}$, $M_{2}$, $T$ and $n$.
\end{theorem}
This mentioned result shows that the time-dependent potential $b$ can also be
stably determined, from the knowledge of the boundary measurements
$\Lambda_{a,b}$ in the same subset $Q_{r,*}\!\subset\! Q$, provided it is
known outside this region. As a consequence, we have the following existence
result
\begin{Corol} Under
the same assumptions of Theorem \ref{Theorem1} and Theorem \ref{Theorem2}, we
have that $\Lambda_{a_{1},b_{1}}=\Lambda_{a_{2},b_{2}}$
 implies  $ a_{1}=a_{2}$  and $b_{1}=b_{2}$ in $Q_{r,*} $.
\end{Corol}
\subsubsection{Determination of coefficients from boundary measurements and final data: } In order to extend the above results to a
larger region $Q_{r,\sharp}\!\supset\!Q_{r,*}$, we require more information
about the solution $u$ of the wave equation \ref{EQ1}. So, in this case we
will add the final data of the solution $u$. This leads to defining the
following  boundary operator (response operator):
$$\begin{array}{rrr}
\mathscr{R}_{a,b}:\mathcal{H}_{0}^{1}(\Sigma)&\longrightarrow&\mathcal{K}:=L^{2}(\Sigma)\times H^{1}(\Omega)\times L^{2}(\Omega)\\
f&\longmapsto&(\p_{\nu}u,u(\cdot,T),\p_{t}u(\cdot,T)).
\end{array}$$
We conclude from (\ref{estim energie}), that $\mathscr{R}_{a,b}$ is a
continuous operator form $\mathcal{H}^{1}_{0}(\Sigma)$ to $\mathcal{K}$. We
denote by $\|\mathscr{R}_{a,b}\|$ its norm in
$\mathcal{L}(\mathcal{H}_{0}^{1}(\Sigma),\mathcal{K}).$
\begin{theorem}\label{Theorem3}
Let $T\!> \!2\!$ Diam$(\Omega)$. There exist constants $C>0$ and
$\mu_{1},\,\mu_{2}\in (0,1)$ such that we have
$$\|a_{2}-a_{1}\|_{L^{\infty}(Q_{r,\sharp})}\leq
 C \Big( \|\mathscr{R}_{a_{1},b_{1}}-\mathscr{R}_{a_{2},b_{2}}\|^{\mu_{1}}+
 |\log \|\mathscr{R}_{a_{1},b_{1}}-\mathscr{R}_{a_{2},b_{2}}\||^{-1} \Big)^{\mu_{2}},$$
for any $(a_{i},b_{i})\in\mathcal{A}(M_{1},M_{2})$ such that
$\|a_{i}\|_{H^{p}(Q)}\leq M_{1}$, for some $p>n/2+3/2$,
$(a_{1},b_{1})=(a_{2},b_{2})$ in $\overline{Q}_{r}\setminus Q_{r,\sharp}$ and
$(\nabla a_{1},\nabla b_{1})=(\nabla a_{2},\nabla  b_{2})$ on $\p Q_{r}\cap\p
Q_{r,\sharp}$. Here $C$ depends only on $\Omega$, $M_{1}$, $M_{2}$, $T$ and
$n$.
\end{theorem}
From the above statement, we can readily derive the following consequence
\begin{theorem}\label{Theorem4}
Let $T\!> \!2\!$ Diam$(\Omega)$. There exist two constants $C>0$ and
$\mu\in (0,1)$ such that we have
$$\|b_{2}-b_{1}\|_{H^{-1}(Q_{r,\sharp})}\leq C \Big( \log |\log \|\mathscr{R}_{a_{2},b_{2}}-\mathscr{R}_{a_{1},b_{1}}\||^{\mu} \Big)^{-1}, $$
for any $(a_{i},b_{i})\in\mathcal{A}(M_{1},M_{2})$ such that
$\|a_{i}\|_{H^{p}(Q)}\leq M_{1}$, for some $p>n/2+3/2$,
$(a_{1},b_{1})=(a_{2},b_{2})$ in $\overline{Q}_{r}\setminus Q_{r,\sharp}$ and
$(\nabla a_{1},\nabla b_{1})=(\nabla a_{2},\nabla b_{2})$ on $\p Q_{r}\cap\p
Q_{r,\sharp}$. Here $C$ depends only on $\Omega$, $M_{1}$, $M_{2}$, $T$ and
$n$.
\end{theorem}
As a consequence, we have the following uniqueness result
\begin{Corol}
Under the same assumptions of Theorem \ref{Theorem3} and Theorem
\ref{Theorem4}, we have that
$\mathscr{R}_{a_{1},b_{1}}=\mathscr{R}_{a_{2},b_{2}}$
 implies  $ a_{1}=a_{2}$  and $b_{1}=b_{2}$ in $Q_{r,\sharp} $.
\end{Corol}
\subsubsection{Determination of coefficients from boundary measurements and final data by varying the initial data:}
In the first and the second case, we can see that there is no hope to recover
the unknown coefficients $a$ and $b$ over the whole domain, since the initial
data $(u_{0},u_{1})$ are  zero. However, we shall prove that this is
no longer the case by considering all possible initial data.

We define the following space
 $\mathcal{F}=H^{1}(\Sigma)\times H^{1}(\Omega)\times
L^{2}(\Omega)$. In this case we will consider observations given by the
following operator:
$$\begin{array}{ccc}
\mathcal{I}_{a,b}:\,\,\,\,\,\,\,\,\,\,\,\,\mathcal{F}&\longrightarrow& \mathcal{K}\\
\,\,\,\,\,\,\,\,\,\,\,\,\,\,\,\,\,\,\,(f,u_{0},u_{1})&\longmapsto& (\p_{\nu}u,u(\cdot,T),\p_{t}u(\cdot,T)).
\end{array}$$
By (\ref{estim energie}), we deduce that $\mathcal{I}_{a,b}$ is continuous
from $\mathcal{F}$ into $\mathcal{K}$, we denote by $\|\mathcal{I}_{a,b}\|$
its norm in $\mathcal{L}(\mathcal{F},\mathcal{K})$. Having said that, we are
now in position to state the last main result.

\begin{theorem}\label{Theorem5}
 There exist constants $C>0$ and $\mu_{1},\,\mu_{2}\in(0,1)$ such that the
 following estimate holds
$$\|a_{2}-a_{1}\|_{L^{\infty}(Q)}\leq
 C \Big( \|\mathcal{I}_{a_{1},b_{1}}-\mathcal{I}_{a_{2},b_{2}}\|^{\mu_{1}}+|\log \|\mathcal{I}_{a_{1},b_{1}}-\mathcal{I}_{a_{2},b_{2}}\||^{-1}
 \Big)^{\mu_{2}},$$
for any $(a_{i},b_{i})\in \mathcal{C}^{2}(\overline{Q})\times
\mathcal{C}^{1}(\overline{Q})$, such that
$\|a_{i}\|_{\mathcal{C}^{2}(Q)}+\|a_{i}\|_{H^{p}(Q)}\leq M_{1}$ for some
$p>n/2+3/2$, $\|b_{i}\|_{\mathcal{C}^{1}(Q)}\leq M_{2}$ and $(\nabla
a_{1},\nabla b_{1})=(\nabla a_{2},\nabla b_{2})$ on $\Sigma$. Here $C$
depends only on $\Omega$, $M_{1}$, $M_{2}$, $T$ and $n$.
\end{theorem}
As a consequence, we have the following result
\begin{theorem}\label{Theorem6}
There exist two constants $C>0$ and $\mu\in (0,1)$ such that the following
estimate holds
$$\|b_{2}-b_{1}\|_{H^{-1}(Q)}\leq C \Big( \log |\log \|\mathcal{I}_{a_{2},b_{2}}-\mathcal{I}_{a_{1},b_{1}}\||^{\mu} \Big)^{-1}, $$
for any $(a_{i},b_{i})\in \mathcal{C}^{2}(\overline{Q})\times
\mathcal{C}^{1}(\overline{Q})$, such that
$\|a_{i}\|_{\mathcal{C}^{2}(Q)}+\|a_{i}\|_{H^{p}(Q)}\leq M_{1}$ for some
$p>n/2+3/2$, $\|b_{i}\|_{\mathcal{C}^{1}(Q)}\leq M_{2}$ and $(\nabla
a_{1},\nabla b_{1})=(\nabla a_{2},\nabla b_{2})$ on $\Sigma$. Here $C$
depends only on $\Omega$, $M_{1}$, $M_{2}$, $T$ and $n$.
\end{theorem}
\begin{Corol}
Under the same assumptions of Theorem \ref{Theorem5} and Theorem
\ref{Theorem6}, we have that
$\mathcal{I}_{a_{1},b_{1}}=\mathcal{I}_{a_{2},b_{2}}$ implies that
$a_{1}=a_{2}$ and $b_{1}=b_{2}$ everywhere in $Q$.
\end{Corol}
The outline of this paper is as follows. Section \ref{Sec2} is devoted to the
construction of geometric optics solutions to the equation (\ref{EQ1}). Using
these particular solutions, we establish in Section \ref{Sec3} stability
estimates for the absorbing coefficient $a$ and the  potential $b$. Section
\ref{Sec4} and \ref{Sec5} are devoted to the proof of the results of the
second and third case respectively. In appendix A, we develop the proof of an
analytic technical result.
\section{Construction of geometric optics solutions}\label{Sec2}
The present section is devoted to the construction of suitable geometrical
optics solutions for the dissipative wave equation (\ref{EQ1}), which are key
ingredients to the proof of our main results. The construction here is a
modification of a similar result in \cite{[R10]}. We shall first state
the following lemma which is needed to prove the main statement of this
section.
\begin{Lemm}(see \cite{[R15]})\label{Lemma2.1}
Let $T,\,M_{1},\,M_{2}>0$, $a\in L^{\infty}(Q)$ and $b\in L^{\infty}(Q)$,
such that $\|a\|_{L^{\infty}(Q)}\leq M_{1}$ and $\|b\|_{L^{\infty}(Q)}\leq
M_{2}$. Assume that $F\in L^{1}(0,T;L^{2}(\Omega))$. Then, there exists  a
unique solution $u$ to the following equation
\begin{equation}\label{EQ2.2}
\left\{
  \begin{array}{ll}
    \p_{t}^{2}u-\Delta u+a(x,t)\p_{t}u+b(x,t)u(x,t)=F(x,t) & \mbox{in}\,\,\,Q, \\
\\
    u(x,0)=0=\p_{t}u(x,0) & \mbox{in}\,\,\,\Omega, \\
\\
    u(x,t)=0 & \mbox{on}\,\,\,\Sigma,
  \end{array}
\right.
\end{equation}
such that
$$u\in \mathcal{C}([0,T];H^{1}_{0}(\Omega))\cap
\mathcal{C}^{1}([0,T];L^{2}(\Omega)),\,\,\,\,  $$
Moreover, there exists a
constant $C>0$ such that
\begin{equation}\label{EQ2.3}
\|\p_{t}u(\cdot,t)\|_{L^{2}(\Omega)}+\|\nabla
u(\cdot,t)\|_{L^{2}(\Omega)}\leq C \|F\|_{L^{1}(0,T;L^{2}(\Omega))}.
\end{equation}
\end{Lemm}
Armed with the above lemma, we may now construct suitable geometrical optics
solutions to the dissipative wave equation (\ref{EQ1}) and to its retrograde
problem. For this purpose, we consider  $\varphi\in \mathcal{C}_{0}^{\infty}(\R^{n})$ and  notice that for 
  all $\omega\in \mathbb{S}^{n-1}$ the function $\phi$ given by
\begin{equation}\label{EQ2.4}
\phi(x,t)=\varphi(x+t\omega),
\end{equation} solves the following transport equation
\begin{equation}\label{EQ2.5}
(\p_{t}-\omega\cdot\nabla)\phi(x,t)=0.
\end{equation}
 We are now in position to prove the following statement
\begin{Lemm}\label{Lemma2.2}
Let $(a_{i},b_{i})\in \mathcal{A}(M_{1},M_{2})$, $i=1,2$. Given $\omega\in
\mathbb{S}^{n-1}$ and $\varphi\in \mathcal{C}_{0}^{\infty}(\R^{n})$, we
consider the function $\phi$ defined by (\ref{EQ2.4}). Then, for any
$\lambda>0$, the following equation
\begin{equation}\label{EQ2.6}
\p_{t}^{2}u-\Delta u
+a(x,t)
\p_{t}u+b(x,t)u=0\,\,\,\,\,\,\,\mbox{in}\,\,\,Q,
\end{equation}
admits a unique solution
$$u^{+}\in \mathcal{C}([0,T];H^{1}(\Omega))\cap \mathcal{C}^{1}([0,T];L^{2}(\Omega)),$$
of the following form
\begin{equation}\label{EQ2.7}
u^{+}(x,t)=\phi(x,t)A^{+}(x,t)e^{i\lambda(x\cdot\omega+t)}+r_{\lambda}^{+}(x,t),
\end{equation}
where $A^{+}(x,t)$ is given by
\begin{equation}\label{EQ2.8}
A^{+}(x,t)=\exp\Big(-\frac{1}{2}\int_{0}^{t}a(x+(t-s)\omega,s)\,ds\Big),
\end{equation}
and   $r_{\lambda}^{+}(x,t)$ satisfies
\begin{equation}\label{EQ2.9}
r_{\lambda}^{+}(x,0)=\p_{t}
r_{\lambda}^{+}(x,0)=0,\,\,\,\mbox{in}\,\,\Omega,\,\,\,\,\,\,r_{\lambda}^{+}(x,t)=0\,\,\,\mbox{on}\,\,\Sigma.
\end{equation}
 Moreover, there exists a positive constant $C>0$ such that
\begin{equation}\label{EQ2.10}
 \lambda\|r_{\lambda}^{+}(\cdot,t)\|_{L^{2}(\Omega)}+\|\p_{t}r_{\lambda}^{+}(\cdot,t)\|_{L^{2}(\Omega)}\leq C\|\varphi\|_{H^{3}(\R^{n})}.
\end{equation}
\end{Lemm}
\begin{proof}{}
We proceed as in the proof of a similar result in \cite{[R10]}. We put
$$g(x,t)=-\Big(\p_{t}^{2}-\Delta +a(x,t) \p_{t}+b(x,t) \Big)\Big(\phi(x,t)A^{+}(x,t)e^{i\lambda(x \cdot\omega +t)}\Big).$$
 In light of (\ref{EQ2.6}) and (\ref{EQ2.7}), it will be enough to prove the existence of $r^{+}=r_{\lambda}^{+}$
satisfying
\begin{equation}\label{EQ2.11}
\left\{
  \begin{array}{ll}
    \Big(\p_{t}^{2}-\Delta +a(x,t) \p_{t}+b(x,t) \Big)r^{+}=g(x,t), &  \\
\\
   r^{+}(x,0)=\p_{t}r^{+}(x,0)=0 , &  \\
\\
    r^{+}(x,t)=0, &
  \end{array}
\right.
\end{equation}
and obeying the  estimate (\ref{EQ2.10}).
 From (\ref{EQ2.5}) and using the fact that $A^{+}(x,t)$ solves the following equation
$$
 2(\p_{t}-2\omega\cdot\nabla) A^{+}(x,t)=-a(x,t) A^{+}(x,t),
$$
we obtain the following identity
$$
g(x,t)=-e^{i\lambda(x\cdot\omega+t)}\Big(\p_{t}^{2}-\Delta +a(x,t)\p_{t}+b(x,t) \Big)\Big(\phi(x,t)A^{+}(x,t) \Big)
=-e^{i\lambda(x\cdot\omega+t)}g_{0}(x,t),
$$
where $g_{0}\in L^{1}(0,T,L^{2}(\Omega))$. Thus, in view of Lemma
\ref{Lemma2.1}, there exists a unique solution
$$r^{+}\in \mathcal{C}([0,T];H^{1}_{0}(\Omega))\cap \mathcal{C}^{1}([0,T];L^{2}(\Omega)),$$
satisfying (\ref{EQ2.11}). Let us now define by $w$ the following function
\begin{equation}\label{EQ2.12}
w(x,t)=\int_{0}^{t}r^{+}(x,s)\,ds.
\end{equation}
We integrate the equation (\ref{EQ2.11}) over $[0,t]$, for $t\in(0,T)$. Then,
in view of (\ref{EQ2.12}), we have
$$\begin{array}{lll}
\Big(\p_{t}^{2}-\Delta +a(x,t) \p_{t}+b(x,t)
\Big)w(x,t)&=&\displaystyle\int_{0}^{t}g(x,s)\,ds+\int_{0}^{t}\Big(b(x,t)-b(x,s)\Big)r^{+}(x,s)\,ds\\
& &+\displaystyle\int_{0}^{t}\p_{s}a(x,s)r^{+}(x,s)\,ds.
\end{array}$$
Therefore,  $w$ is a solution to the following equation
$$
\left\{
  \begin{array}{ll}
   \Big( \p_{t}^{2}-\Delta +a(x,t) \p_{t}+b(x,t)\Big)w(x,t)=F_{1}(x,t)+F_{2}(x,t)  & \mbox{in}\,\,\,Q, \\
\\
    w(x,0)=0=\p_{t}w(x,0) & \mbox{in}\,\,\,\Omega, \\
\\
    w(x,t)=0 & \mbox{on}\,\,\,\Sigma,
  \end{array}
\right.
$$
where $F_{1}$ and $F_{2}$ are given by
\begin{equation}\label{EQ2.13}
F_{1}(x,t)=\int_{0}^{t}g(x,s)\,ds,
\end{equation}
and
$$F_{2}(x,t)=\int_{0}^{t}\Big(b(x,t)-b(x,s)\Big)r^{+}(x,s)\,ds+\int_{0}^{t}\p_{s}a(x,s)r^{+}(x,s)\,ds.$$
Let $\tau\in [0,T]$. Applying Lemma \ref{Lemma2.1} on the interval
$[0,\tau]$, we get
$$\begin{array}{lll}
\|\p_{t}w(\cdot,\tau)\|^{2}_{L^{2}(\Omega)}
&\leq& C \Big( \|F_{1}\|_{L^{2}(0,T;L^{2}(\Omega))}^{2}+T \,\para{M_{1}^{2}+4M_{2}^{2}}\displaystyle\int_{0}^{\tau}\!\!
\displaystyle\int_{\Omega}\!\displaystyle\int_{0}^{t}|r^{+}(x,s)|^{2}\,ds\,dx\,dt \Big).
\end{array}$$
From (\ref{EQ2.12}), we get
$$\begin{array}{lll}
\|\p_{t}w(\cdot,\tau)\|^{2}_{L^{2}(\Omega)}&\leq& C
 \Big(\|F_{1}\|_{L^{2}(0,T;L^{2}(\Omega))}^{2}+\displaystyle\int_{0}^{\tau}\!\!\!\int_{0}^{t}\|\p_{s}w(\cdot,s)\|_{L^{2}(\Omega)}^{2}\,ds\,dt
 \Big)\\
&\leq& C\Big( \|F_{1}\|^{2}_{L^{2}(0,T;L^{2}(\Omega))}+T \displaystyle\int_{0}^{\tau}\|\p_{s}w(\cdot,s)\|^{2}_{L^{2}(\Omega)}\,ds \Big).
\end{array}$$
Therefore, from Gronwall's Lemma, we find out that
$$\|\p_{t}w(\cdot,\tau)\|_{L^{2}(\Omega)}^{2}\leq C \|F_{1}\|^{2}_{L^{2}(0,T;L^{2}(\Omega))}.$$
As a consequence, in light of (\ref{EQ2.12}), we conclude that
$\|r^{+}(\cdot,t)\|_{L^{2}(\Omega)}\leq C
\|F_{1}\|_{L^{2}(0,T;L^{2}(\Omega))}.$ Further, according to (\ref{EQ2.13}),
$F_{1}$ can be written as follows
$$F_{1}(x,t)=\int_{0}^{t}g(x,s)\,ds=\frac{1}{i\lambda} \int_{0}^{t}\,g_{0}(x,s)\p_{s}(e^{i\lambda(x\cdot\omega+s)})\,ds.$$
Integrating by parts with respect to $s$, we conclude that there exists a
positive constant $C>0$ such that
$$\|r^{+}(\cdot,t)\|_{L^{2}(\Omega)}\leq \frac{C}{\lambda}\|\varphi\|_{H^{3}(\R^{n})}.$$
Finally, since $\|g\|_{L^{2}(0,T;L^{2}(\Omega))}\leq C \|\varphi\|_{H^{3}(\R^{n})}$ , the
energy estimate (\ref{EQ2.3}) associated to  the problem (\ref{EQ2.11})
yields
$$\|\p_{t}r^{+}(\cdot,t)\|_{L^{2}(\Omega)}+\|\nabla r^{+}(\cdot,t)\|_{L^{2}(\Omega)}\leq C \|\varphi\|_{H^{3}(\R^{n})}. $$
This completes the proof of the lemma.
\end{proof}
As a consequence we have the following lemma
\begin{Lemm}\label{Lemma2.3}
Let $(a_{i},b_{i})\in \mathcal{A}(M_{1},M_{2})$, $i=1,\,2$. Given $\omega\in
\mathbb{S}^{n-1}$ and $\varphi\in \mathcal{C}_{0}^{\infty}(\R^{n})$, we
consider  the function $\phi$  defined by (\ref{EQ2.4}). Then, the following
equation
\begin{equation}\label{EQ2.14}
\p_{t}^{2}u-\Delta u -a(x,t) \p_{t}u+b(x,t)u=0\,\,\,\,\,\mbox{in}\,\,\,Q,
\end{equation}
admits a unique solution
$$u^{-}\in \mathcal{C}([0,T];H^{1}(\Omega))\cap \mathcal{C}^{1}([0,T];L^{2}(\Omega)),$$
of the following form
\begin{equation}\label{EQ2.15}
u^{-}(x,t)=\varphi(x+t\omega)A^{-}(x,t)e^{-i\lambda(x\cdot\omega+t)}+{r}_{\lambda}^{-}(x,t),
\end{equation}
where ${A}^{-}(x,t)$ is given by
\begin{equation}\label{EQ2.16}
{A}^{-}(x,t)=\exp\Big(\frac{1}{2}\int_{0}^{t}a(x+(t-s)\omega,s)\,ds\Big),
\end{equation}
and ${r}_{\lambda}^{-}(x,t)$ satisfies
\begin{equation}\label{EQ2.17}
r_{\lambda}^{-}(x,T)=\p_{t}
{r}_{\lambda}^{-}(x,T)=0,\,\,\,\mbox{in}\,\,\,\Omega,\,\,\,\,\,\,{r}_{\lambda}^{-}(x,t)=0\,\,\,\mbox{on}\,\,\,\Sigma.
\end{equation}
Moreover, there exists a constant $C>0$ such that
\begin{equation}\label{EQ2.18}
\lambda\|{r}_{\lambda}^{-}(\cdot,t)\|_{L^{2}(\Omega)}+\|\p_{t} {r}_{\lambda}^{-}(\cdot,t)\|_{L^{2}(\Omega)}\leq C \|\varphi\|_{H^{3}(\R^{n})}.
\end{equation}
\end{Lemm}
\begin{proof}{}
We prove this result by proceeding  as in the proof of Lemma \ref{Lemma2.2}.
Putting
$$\widetilde{g}(x,t)=-\Big(\p_{t}^{2}-\Delta -a(x,t)
\p_{t}+b(x,t) \Big)\Big(\phi(x,t){A}^{-}(x,t)e^{-i\lambda(x
\cdot\omega +t)}\Big).$$ Then, it would be enough to see that if ${r}^{-}=r_{\lambda}^{-}$
is solution to the following system
$$\begin{array}{lll}
\left\{
  \begin{array}{ll}
    \Big(\p_{t}^{2}-\Delta -a(x,t)\p_{t}+b(x,t)  \Big){r}^{-}(x,t)=\widetilde{g}(x,t) & \mbox{in}\,\,\,Q, \\
\\
    {r}^{-}(x,T)=0=\p_{t}{r}^{-}(x,T) & \mbox{in}\,\,\,\Omega,\\
\\
    {r}^{-}(x,t)=0& \mbox{on}\,\,\,\Sigma,\\
  \end{array}
\right.
\end{array}$$
then, $r^{+}(x,t)={r}^{-}(x,T-t)$ is a solution to (\ref{EQ2.11}) with
$g(x,t)=\widetilde{g}(x,T-t)$ and $a(x,t)$, $b(x,t)$ are replaced by
$a(x,T-t)$ and  $b(x,T-t)$.
\end{proof}
\section{ Determination of coefficients from boundary measurements}\label{Sec3}
In this section we prove  stability estimates for the absorbing coefficient
$a$ and the potential $b$ appearing in the initial boundary value problem
(\ref{EQ1}) by the use of the geometrical optics solutions constructed in
section \ref{Sec2} and the light-ray transform. We assume that Supp
$\varphi\subset \mathcal{A}_{r}$, in such a way we have
$$\mbox{Supp}\varphi\cap \Omega=\emptyset\,\,\,\,\,\mbox{and}\,\,\,\,\,(\mbox{Supp}\varphi \pm T\omega)\cap \Omega=\emptyset,\,\,\,
\forall\,\omega\in \mathbb{S}^{n-1}.$$
\subsection{Stability for the absorbing coefficient}\label{subsection3.1}
The present section is devoted to the proof of Theorem \ref{Theorem1}. Our
goal here is to show that the time dependent coefficient $a$ depends stably
on the Dirichlet-to-Neumann map $\Lambda_{a,b}$. In the rest of this section,
we define $a$ in $\R^{n+1}$ by $a=a_{2}-a_{1}$ in $\overline{Q}_{r}$ and
$a=0$ on $\R^{n+1}\setminus \overline{Q}_{r}$. We start by collecting a
preliminary estimate which is needed to prove the main statement of this
section.
\subsubsection{Preliminary estimate}
The main purpose of this section is to give a preliminary estimate, which
relates the difference of the absorbing coefficients to the
Dirichlet-to-Neumann map. Let $\omega\in\mathbb{S}^{n-1}$, and
$(a_{i},b_{i})\in\mathcal{A}$ such that $(a_{1},b_{1})=(a_{2},b_{2})$ in
$\overline{Q}_{r}\setminus Q_{r,*}$. We set
$$a=a_{2}-a_{1},\,\,\,\,\,b=b_{2}-b_{1},\,\,\mbox{and}\,\,A(x,t)
=(A^{-}A^{+})(x,t)=\exp\Big(-\displaystyle\frac{1}{2}\displaystyle\int_{0}^{t}a(x+(t-s)\omega,s)\,ds\Big).$$
Here, we recall the definition of $A^{-}$ and $A^{+}$
$$A^{-}(x,t)=\exp\Big( \frac{1}{2}\int_{0}^{t}a_{1}(x+(t-s)\omega,s)\,ds \Big),\,\,\,\,
A^{+}(x,t)=\exp\Big(-\frac{1}{2}\int_{0}^{t}a_{2}(x+(t-s)\omega,s)\,ds\Big).$$
The main result of this section can be stated as follows
\begin{Lemm}\label{Lemma3.1}
Let $(a_{i},b_{i})\in \mathcal{A}(M_{1},M_{2})$, $i=1,\,2$. There exists
$C>0$ such that for any $\omega\in\mathbb{S}^{n-1}$ and $\varphi\in
\mathcal{C}_{0}^{\infty}(\mathcal{A}_{r})$, the following estimate holds true
$$\Big|\int_{\R^{n}}\varphi^{2}(y)\Big[ \exp\Big( -\frac{1}{2}\int_{0}^{T}a(y-s\omega,s)\,ds \Big)-1\Big]\,dy  \Big|\leq C\Big(\lambda^{2}
\|\Lambda_{a_{2},b_{2}}-\Lambda_{a_{1},b_{1}}\|+\frac{1}{\lambda}
\Big)\|\varphi\|_{H^{3}(\R^{n})}^{2},$$ for any sufficiently large $\lambda>0$.  Here $C$ depends only on
$\Omega$, $T$, $M_{1}$ and $M_{2}$.
\end{Lemm}
\begin{proof}{}
In view of Lemma \ref{Lemma2.2}, and using the fact that Supp
$\varphi\cap\Omega=\emptyset$, there exists a geometrical optics solution
$u^{+}$ to the equation
$$
\left\{
    \begin{array}{ll}
      \p_{t}^{2}u^{+}-\Delta u^{+}+a_{2}(x,t)\p_{t}u^{+}+b_{2}(x,t)u^{+}=0 & \mbox{in}\,\,\,Q, \\
\\
      u^{+}(x,0)=\p_{t}u^{+}(x,0)=0 & \mbox{in}\,\,\,\Omega,
    \end{array}
  \right.$$
in the following form
\begin{equation}\label{EQ3.19}
u^{+}(x,t)=\varphi(x+t\omega)A^{+}(x,t)e^{i\lambda(x\cdot\omega+t)}+r_{\lambda}^{+}(x,t),
\end{equation}
corresponding to the coefficients $a_{2}$ and $b_{2}$,  where $r^{+}(x,t)$
satisfies (\ref{EQ2.9}), (\ref{EQ2.10}). Next, let us denote by $f_{\lambda}$
the function
$$f_{\lambda}(x,t)= u^{+}(x,t)_{|\Sigma}=\varphi(x+t\omega) A^{+}(x,t)e^{i\lambda(x\cdot\omega+t)}.$$
We denote by $u_{1}$ the solution of
$$\left\{
  \begin{array}{ll}
    \p_{t}^{2}u_{1}-\Delta u_{1}+a_{1}(x,t)\p_{t}u_{1}+b_{1}(x,t)u_{1}=0 & \mbox{in}\,\,\,Q, \\
\\
    u_{1}(x,0)=\p_{t}u_{1}(x,0)=0 & \mbox{in}\,\,\,\Omega, \\
\\
    u_{1}=f_{\lambda} & \mbox{on}\,\,\,\Sigma.
  \end{array}
\right.
$$
Putting $u=u_{1}-u^{+}$. Then, $u$ is a solution to the following system
\begin{equation}\label{EQ3.20}
\left\{
  \begin{array}{ll}
    \p_{t}^{2}u-\Delta u +a_{1}(x,t) \p_{t}u+b_{1}(x,t)u=a(x,t)\p_{t}u^{+}+b(x,t)u^{+} & \mbox{in}\,\,\,Q, \\
\\
    u(x,0)=\p_{t}u(x,0)=0 & \mbox{in}\,\,\,\Omega, \\
\\
    u(x,t)=0 & \mbox{on}\,\,\,\Sigma.
  \end{array}
\right.
\end{equation}
 where $a=a_{2}-a_{1}$ and $b=b_{2}-b_{1}$. On the other hand Lemma \ref{Lemma2.3} and the fact that (Supp $\varphi\pm
 T\omega)\cap\Omega=\emptyset$,
  guarantee the existence of a
geometrical optic solution $u^{-}$ to the backward problem of (\ref{EQ1})
$$
\left\{
  \begin{array}{ll}
    \p_{t}^{2}u^{-}-\Delta u^{-} -a_{1}(x,t)\p_{t}u^{-}+(b_{1}(x,t)-\p_{t} a_{1}(x,t)) u^{-} =0 & \mbox{in}\,\,\,Q, \\
\\
    u^{-}(x,T)=0=\p_{t}u^{-}(x,T) & \mbox{in}\,\,\,\,\Omega,
  \end{array}
\right.
$$
corresponding to the coefficients $a_{1}$ and $(-\p_{t} a_{1}+b_{1})$, in the
form
\begin{equation}\label{EQ3.21}
u^{-}(x,t)=\varphi(x+t\omega)e^{-i\lambda(x\cdot\omega+t)}A^{-}(x,t)+r_{\lambda}^{-}(x,t),
\end{equation}
 where $r_{\lambda}^{-}(x,t)$
satisfies (\ref{EQ2.17}), (\ref{EQ2.18}). Multiplying the first equation of
(\ref{EQ3.20}) by $u^{-}$, integrating by parts and using Green's formula, we
obtain
\begin{eqnarray}\label{EQ3.22}
\,\,\,\,\,\,\,\,\,
\displaystyle\int_{0}^{T}\!\!\!\int_{\Omega}a(x,t)
\p_{t}u^{+}\,u^{-}\,dx\,dt+\displaystyle\int_{0}^{T}\!\!\!\int_{\Omega}b(x,t)u^{+}\,u^{-}\,dx\,dt
&=&\displaystyle\int_{0}^{T}\!\!\!\int_{\Gamma} (\Lambda_{a_{2},b_{2}}-\Lambda_{a_{1},b_{1}}) (f_{\lambda}) \,u^{-}\,d\sigma\,dt.
\end{eqnarray}
On the other hand, by replacing $u^{+}$ and $u^{-}$ by their expressions, we
have
$$\begin{array}{lll}
&&\displaystyle\int_{0}^{T}\!\!\!\int_{\Omega}
a(x,t)\p_{t}u^{+}\,u^{-}\,dx\,dt=\displaystyle\int_{0}^{T}\!\!\!\int_{\Omega}a(x,t)\p_{t}\varphi(x+t\omega)e^{i\lambda(x\cdot\omega+t)}A^{+}
r_{\lambda}^{-}\,dx\,dt\\
&&+\displaystyle\int_{0}^{T}\!\!\!\int_{\Omega}a(x,t)\varphi(x+t\omega)e^{i\lambda(x\cdot\omega+t)}\p_{t}A^{+}r_{\lambda}^{-}\,dx\,dt
+\displaystyle\int_{0}^{T}\!\!\!\int_{\Omega}\!\!a(x,t)\p_{t}\varphi(x+t\omega)\varphi(x+t\omega)(A^{+}A^{-})dx\,dt\\
&&\,\,\,\,\,\,\,\,\,\,\,\,+\displaystyle\int_{0}^{T}\!\!\!\int_{\Omega}a(x,t)\varphi^{2}(x+t\omega) \p_{t}A^{+}A^{-}\,dx\,dt
+\,i\lambda \displaystyle\int_{0}^{T}\!\!\!\int_{\Omega}a(x,t) \varphi(x+t\omega)e^{i\lambda(x\cdot\omega+t)}A^{+}
r_{\lambda}^{-}\,dx\,dt\\
&&\,\,\,\,\,\,\,\,\,\,\,\,\,\,\,\,\,\,+\displaystyle\int_{0}^{T}\!\!\!\int_{\Omega}a(x,t)\varphi(x+t\omega)e^{-i\lambda(x.\omega+t)}A^{-}
\p_{t}r_{\lambda}^{+}dx\,dt
+i\lambda\displaystyle\int_{0}^{T}\!\!\!\int_{\Omega}a(x,t)\varphi^{2}  (x+t\omega)(A^{+}A^{-})\,dx\,dt\\
&&\,\,\,\,\,\,\,\,\,\,\,\,\,\,\,\,\,\,\,\,\,\,\,\,\,\,\,\,\,\,\,\,\,\,\,\,\,\,\,\,\,+\displaystyle\int_{0}^{T}\!\!\!\int_{\Omega}
a(x,t)\p_{t}r_{\lambda}^{+}r_{\lambda}^{-}\,dx\,dt
=i\lambda\displaystyle\int_{0}^{T}\!\!\!\int_{\Omega}a(x,t)\varphi^{2}(x+t\omega)A\,dx\,dt+\mathcal{I}_{\lambda},
\end{array}$$
where $A=A^{+}A^{-}$. In light of (\ref{EQ3.22}), we have
\begin{equation}\label{EQ3.23}
i\lambda\displaystyle\int_{0}^{T}\!\!\!\int_{\Omega} a(x,t)\varphi^{2}(x+t\omega)A(x,t)\,dx\,dt
=\displaystyle\int_{0}^{T}\!\!\!\int_{\Gamma}(\Lambda_{a_{2},b_{2}}-\Lambda_{a_{1},b_{1}})(f_{\lambda})\,u^{-}\,d\sigma\,dt
-\int_{0}^{T}\!\!\!\int_{\Omega}b(x,t) u^{+}u^{-}\,dx\,dt
-\mathcal{I}_{\lambda}.
\end{equation}
Note that for $\lambda$ sufficiently large, we have
\begin{equation}\label{EQ3.24}
\|u^{+}u^{-}\|_{L^{1}(Q)}\leq C \|\varphi\|^{2}_{H^{3}(\R^{n})},\,\,\,\mbox{and}\,\,\,\,\,\,\,\,|\mathcal{I}_{\lambda}|\leq C
\|\varphi\|_{H^{3}(\R^{n})}^{2}.
\end{equation}
On the other hand, since on $\Sigma$, we have $u^{+}=f_{\lambda}$ and
$r^{-}_{\lambda}=r^{+}_{\lambda}=0$, then, we get the following estimate
\begin{eqnarray}\label{EQ3.25}
\Big|\displaystyle\int_{0}^{T}\!\!\!\int_{\Gamma}(\Lambda_{a_{2},b_{2}}-\Lambda_{a_{1},b_{1}})(f_{\lambda})\,u^{-}\,d\sigma\,dt\Big|&\leq&
\|\Lambda_{a_{2},b_{2}}-\Lambda_{a_{1},b_{1}}\|\,\|f_{\lambda}\|_{H^{1}(\Sigma)}\|u^{-}\|_{L^{2}(\Sigma)}\cr
\cr
&\leq& \|\Lambda_{a_{2},b_{2}}-\Lambda_{a_{1},b_{1}}\|\,\|u^{+}-r_{\lambda}^{+}\|_{H^{2}(Q)}\,\|u^{-}-r_{\lambda}^{-}\|_{H^{1}(Q)}\cr
\cr
&\leq& C\lambda^{3}\|\Lambda_{a_{2},b_{2}}-\Lambda_{a_{1},b_{1}}\|\,\|\varphi\|_{H^{3}(\R^{n})}^{2}.
\end{eqnarray}
Consequently, by (\ref{EQ3.23}), (\ref{EQ3.24}) and (\ref{EQ3.25}), we obtain
$$\Big|\displaystyle\int_{0}^{T}\!\!\!\int_{\Omega}
a(x,t)\varphi^{2}(x+t\omega)A(x,t)\,dx\,dt \Big|\leq C\Big(\lambda^{2}
\|\Lambda_{a_{2},b_{2}}-\Lambda_{a_{1},b_{1}}\|+\frac{1}{\lambda}
\Big)\|\varphi\|_{H^{3}(\R^{n})}^{2},$$
 where $A(x,t)=\exp\Big(-\displaystyle\frac{1}{2}\displaystyle\int_{0}^{t}a(x+(t-s)\omega,s)\,ds\Big)$.
Then, using the fact $a(x,t)=0$ outside $Q_{r,*}$ and making this change of
variables $y=x+t\omega$, one gets the following estimation
$$\Big|\int_{0}^{T}\!\!\!\int_{\R^{n}} a(y-t\omega,t)\varphi^{2}(y)\exp\Big(-\frac{1}{2}\int_{0}^{t}a(y-s\omega,s)\,ds   \Big)\,dy\,dt  \Big|\leq
C\Big(\lambda^{2}
\|\Lambda_{a_{2},b_{2}}-\Lambda_{a_{1},b_{1}}\|+\frac{1}{\lambda}
\Big)\|\varphi\|_{H^{3}(\R^{n})}^{2}.$$
Bearing in mind that
$$\begin{array}{lll}
\displaystyle\int_{0}^{T}\!\!\!\int_{\R^{n}}a(y-t\omega,t)\!\!\!&\varphi^{2}(y)&\!\!\!
\exp\Big( -\displaystyle\frac{1}{2}\int_{0}^{t}a(y-s\omega,s)\,ds\Big)\,dy\,dt\cr
&=&\!\!\!\!\!\!\!\!\!\!\!\!-2\displaystyle\int_{0}^{T}\!\!\!\int_{\R^{n}}\varphi^{2}(y)
\frac{d}{dt}\Big[ \exp\Big( -\displaystyle\frac{1}{2}\int_{0}^{t}a(y-s\omega,s)\,ds  \Big)
\Big]\,dy\,dt\cr &=&\!\!\!\!\!\!\!\!\!\!\!\!-2\displaystyle\int_{\R^{n}}\varphi^{2}(y)\Big[
\exp \Big( -\frac{1}{2}\int_{0}^{T}a(y-s\omega,s)\,ds  \Big)-1\Big]\,dy,
\end{array}$$
we conclude the desired estimate given by
$$\Big|\int_{\R^{n}}\varphi^{2}(y)\Big[ \exp\Big( -\frac{1}{2}\int_{0}^{T}a(y-s\omega,s)\,ds \Big)-1 \Big]\,dy  \Big|\leq C\Big(\lambda^{2}
\|\Lambda_{a_{2},b_{2}}-\Lambda_{a_{1},b_{1}}\|+\frac{1}{\lambda}
\Big)\|\varphi\|_{H^{3}(\R^{n})}^{2}.$$
  This completes the proof of the lemma.
\end{proof}

\subsubsection{Stability for the light-ray transform} The light-ray transform
$\mathcal{R}$ maps a function $f\in L^{1}(\R^{n+1})$ defined in $\R^{n+1}$
into the set of its line integrals. More precisely, if $\omega\in
\mathbb{S}^{n-1}$ and $(x,t)\in \R^{n+1}$, the function
$$\mathcal{R}(f)(x,\omega):=\int_{\R}f(x-t\omega,t)\,dt,$$
is the integral of $f$ over the lines $\{(x-t\omega,t),\,\,t\in\R\}$. The
goal in this section is to obtain an estimate that links the light-ray
transform of the absorbing coefficient $a=a_{2}-a_{1}$ to the measurement
$\Lambda_{a_{2},b_{2}}-\Lambda_{a_{1},b_{1}}$ on a precise set. Using the
above lemma, we can control the light-ray transform of $a$ as follows:
\begin{Lemm}\label{Lemma3.2}
 Let $(a_{i},b_{i})\in \mathcal{A}(M_{1},M_{2})$, $i=1,\,2$. There exist $C>0,$
$\delta>0$, $\beta>0$ and $\lambda_{0}>0$ such that for all $\omega\in
\mathbb{S}^{n-1},$ we have
$$|\mathcal{R}(a)(y,\omega)|\leq C \Big( \lambda^{\delta}\|\Lambda_{a_{2},b_{2}}-\Lambda_{a_{1},b_{1}}\|+\frac{1}{\lambda^{\beta}}  \Big)
,\,\,\,\,\,\,\mbox{a.e.}\, y\in\R^{n},$$
for any $\lambda\geq\lambda_{0}$. Here $C$ depends only on $\Omega$, $T$,
$M_{1}$ and $M_{2}$.
\end{Lemm}
\begin{proof}{}
Let $\psi\in\mathcal{C}_{0}^{\infty}(\R^{n})$ be a positive function which is
supported in the unit ball $B(0,1)$ and such that
$\|\psi\|_{L^{2}(\R^{n})}=1$. Define
\begin{equation}\label{EQ3.26}
\varphi_{h}(x)=h^{-n/2}\psi\Big(\frac{x-y}{h}\Big),
\end{equation} where $y\in\mathcal{A}_{r}$. Then, for $h>0$ sufficiently small we can verify that
$$
\mbox{Supp}\,\varphi_{h}\cap\Omega=\emptyset,\,\,\,\,\,\mbox{and}\,\,\,\,\mbox{Supp}\,\varphi_{h}\pm T\omega\cap \Omega=\emptyset.
$$
Moreover, we have
\begin{eqnarray}\label{EQ3.27}
\!\!\!\!\!\!\!\!\!&\Big|&\!\!\!\exp\Big[-\displaystyle\frac{1}{2}\displaystyle\int_{0}^{T}a(y-s\omega,s)\,ds
\Big] -1 \Big|=\Big|
\displaystyle\int_{\R^{n}}\varphi_{h}^{2}(x)\Big[ \exp\Big(-\frac{1}{2} \int_{0}^{T}a(y-s\omega,s)\,ds \Big)-1  \Big]\,dx\Big|\cr
&&\leq\,\,\,\,\,\,\,\,\,\,\Big| \displaystyle\int_{\R^{n}}\varphi_{h}^{2}(x)\Big[\exp\Big( -\frac{1}{2}
\displaystyle\int_{0}^{T}a(y-s\omega,s)\,ds \Big)
-\exp\Big(-\displaystyle\frac{1}{2}\int_{0}^{T}a(x-s\omega,s)\,ds\Big)   \Big]  \,dx\Big|\cr
 &&\,\,\,\,\,\,\,\,\,\,\,\,\,\,\,\,\,\,\,\,\,\,\,\,\,\,\,\,\,\,\,\,\,\,\,\,\,\,\,\,\,\,\,\,\,\,\,\,\,\,\,\,\,
 \,\,\,\,\,
 +\Big| \displaystyle\int_{\R^{n}} \varphi_{h}^{2}(x)\Big[\exp\Big(-\displaystyle\frac{1}{2}\int_{0}^{T}a(x-s\omega,s)\,ds  \Big)-1 \Big]
 dx\Big|.
\end{eqnarray}
Therefore, since we have
$$\begin{array}{lll} \Big| \!\exp\Big(
\!-\displaystyle\frac{1}{2}\!\int_{0}^{T}\!\!a(y-s\omega,s)ds \Big)
\!-\!\exp\Big(\! -\displaystyle\frac{1}{2}\!\int_{0}^{T}\!\!a(x-s\omega,s)ds
\! \Big) \Big|\!\!\!&\leq&\!\!\! C\Big|\displaystyle\int_{0}^{T}\!\!
a(y-s\omega,s) \!-\!a(x-s\omega,s)ds\Big|,\cr
\end{array}$$
and using the fact that
$$
\Big|\displaystyle\int_{0}^{T}\!\! \Big( a(y-s\omega,s)\!-\!a(x-s\omega,s)\Big) ds\Big|\leq C \,|y-x|,
$$
 we deduce upon applying Lemma \ref{Lemma2.3} with $\varphi=\varphi_{h}$ the following estimation
$$
\Big|\exp\Big(-\frac{1}{2}\int_{0}^{T}a(y-s\omega,s)\,ds  \Big)-1
\Big|\leq C\int_{\R^{n}}\varphi_{h}^{2}(x)\,|y-x|\,dx+C\Big(
\lambda^{2}\|\Lambda_{a_{2},b_{2}}-\Lambda_{a_{1},b_{1}}\|+\frac{1}{\lambda}
\Big)\|\varphi_{h}\|^{2}_{H^{3}(\R^{n})}.
$$
 On the other hand, we have
$$
\|\varphi_{h}\|_{H^{3}(\R^{n})}\leq C h^{-3}\,\,\,\,\,\,\,\mbox{and}\,\,\,\,\int_{\R^{n}}\varphi_{h}^{2}(x)|y-x|\,dx \leq C h.
$$
So that we end up getting the following inequality
$$
 \Big|\exp\Big(-\frac{1}{2}\int_{0}^{T}a(y-s\omega,s)\,ds  \Big)-1
\Big|\leq C \,h+C\Big(
\lambda^{2}\|\Lambda_{a_{2},b_{2}}-\Lambda_{a_{1},b_{1}}\|+\frac{1}{\lambda}
\Big)h^{-6}.
 $$
 Selecting $h$ small such that $h=1/\lambda h^{6}$,
that is $h=\lambda^{-1/7}$, we find two constants $\delta>0$ and $\beta>0$
such that
$$\Big| \exp \Big( -\frac{1}{2}\int_{0}^{T}a(y-s\omega,s)\,ds \Big)-1 \Big|\leq
C \Big[
\lambda^{\delta}\|\Lambda_{a_{2},b_{2}}-\Lambda_{a_{1},b_{1}}\|+\frac{1}{\lambda^{\beta}}
\Big].$$
 Now, using the fact that $|X|\leq e^{M}\,|e^{X}-1|$ for any $|X|\leq M$, we
deduce that
$$\Big| -\frac{1}{2}\int_{0}^{T}a(y-s\omega,s)\,ds \Big|\leq e^{M_{1}T}\Big| \exp\Big( -\frac{1}{2}\int_{0}^{T}a(y-s\omega,s)\,ds  \Big)-1
\Big|.$$ Hence, we conclude that for all $y\in \mathcal{A}_{r}$ and
$\omega\in\mathbb{S}^{n-1}$ we have
$$\Big|\int_{0}^{T}a(y-s\omega,s)\,ds \Big|\leq C \Big(
\lambda^{\delta}\|\Lambda_{a_{2},b_{2}}-\Lambda_{a_{1},b_{1}}\|+\frac{1}{\lambda^{\beta}}
\Big).$$
 Since $a=a_{2}-a_{1}=0$ outside $Q_{r,*}$, this entails that for all $y\in
 \mathcal{A}_{r}$, and $\,\omega\in\mathbb{S}^{n-1}$, we have
\begin{equation}\label{EQ3.28}
\left|\int_{\R}a(y-t\omega,t)\,dt\right|\leq\,C\,\Big(\lambda^{\delta}\|\Lambda_{a_{2},b_{2}}-\Lambda_{a_{1},b_{1}}\|+\frac{1}{\lambda^{\beta}}\Big).
\end{equation}
Moreover, if  $y\in B(0,r/2)$, we have $ |y-t\omega|\geq|t|-|y|\geq
|t|-\displaystyle\frac{r}{2}.$ Hence, one can see that $(y-t\omega,t)\notin
\mathscr{C}_{r}^{+}$ if $t>r/2$. On the other hand, we have
$(y-t\omega,t)\notin \mathscr{C}_{r}^{+}$ if $t\leq
\displaystyle\frac{r}{2}$. Thus, we conclude that $(y-t\omega,t)\notin
\mathscr{C}_{r}^{+}\supset Q_{r,*}$ for $t\in\R.$ This and the fact that
$a=a_{2}-a_{1}=0$ outside $Q_{r,*}$, entails that for all $y\in B(0,r/2)$ and
$\omega\in\mathbb{S}^{n-1}$, we have
$$
a(y-t\omega,t)=0,\,\,\forall\,t\in\R.
$$
By a similar way, we prove for $|y|\geq T-r/2$, that
$(y-t\omega,t)\notin\mathscr{C}_{r}^{-}\supset Q_{r,*}$ for $t\in\R$ and then
$a(y-t\omega,t)=0$. Hence, we conclude that
\begin{equation}\label{EQ3.29}
\int_{\R}a(y-t\omega,t)\,dt=0,\,\,\,\,\mbox{a.e.}\,\,y\notin\mathscr{A}_{r},\,\,\,\omega\in
\mathbb{S}^{n-1}.
\end{equation}
Thus, by (\ref{EQ3.28}) and (\ref{EQ3.29}) we finish the proof of the lemma
by getting
$$
|\mathcal{R}(a)(y,\omega)|=\left|\int_{\R}a(t,y-t\omega)\,dt\right|\leq\,
C\,\Big(\lambda^{\delta}\|\Lambda_{a_{2},b_{2}}-\Lambda_{a_{1},b_{1}}\|
+\frac{1}{\lambda^{\beta}}\Big),
\,\,\,\mbox{a.e}.\,\,y\in\R^{n},\,\,\,\omega\in\mathbb{S}^{n-1}.
$$
The proof of Lemma \ref{Lemma3.2} is complete.
\end{proof}

 Our goal  is to
obtain an estimate linking the Fourier transform with respect to $(x,t)$ of
the absorbing coefficient $a=a_{2}-a_{1}$ to the measurement
$\Lambda_{a_{2},b_{2}}-\Lambda_{a_{1},b_{1}}$ in this set
\begin{equation}\label{EQ3.30}
E=\{(\xi,\tau)\in\para{\R^{n}\setminus{\{0_{\R^{n}}\}}}\times\R,\,\,\,|\tau|<|\xi|\}.
\end{equation}
We denote by $\widehat{F}$ the Fourier transform of a function $F\in
L^{1}(\R^{n+1})$ with respect to $(x,t)$:
$$
\widehat{F}(\xi,\tau)=\int_{\R}\int_{\R^{n}}F(x,t){e^{-ix\cdot\xi}}e^{-it\tau}\,dx\,dt.
$$
We aim for proving that the Fourier transform of $a$  is bounded as follows:
\begin{Lemm}\label{Lemma3.3}
Let $(a_{i},b_{i})\in \mathcal{A}(M_{1},M_{2})$, $i=1,2$. There exist $C>0$,
$\delta>0,$ $\beta>0$ and $\lambda_{0}>0$, such that the following estimate
$$
|\widehat{a}(\xi,\tau)|\leq C\,\Big(\lambda^{\delta}\|\Lambda_{a_{2},b_{2}}-\Lambda_{a_{1},b_{1}}\|+\frac{1}{\lambda^{\beta}}\Big),
$$
holds for any $(\xi,\tau)\in E$ and $\lambda\geq \lambda_{0}$.
\end{Lemm}
\begin{proof}{}
Let $(\xi,\tau)\in E$ and $\zeta\in\mathbb{S}^{n-1}$ be such that
$\xi\cdot\zeta=0$. Setting
$$
\omega=\frac{\tau}{|\xi|^{2}}\cdot\xi+\sqrt{1-\frac{\tau^{2}}{|\xi|^{2}}}\cdot\zeta.
$$
Then, one can see that  $\omega\in\mathbb{S}^{n-1}$ and
${\omega\cdot\xi=\tau.}$ On the other hand  by the change of variable
$x=y-t\omega$, we have  for all $\xi\in\R^{n}$ and
$\omega\in\mathbb{S}^{n-1}$ the following identity
$$\begin{array}{lll}
\displaystyle\int_{\R^{n}}\mathcal{R}(a)(y,\omega)\,{e^{-iy\cdot\xi}}\,dy&=&\displaystyle\int_{\R^{n}}\displaystyle\Big(\int_{\R}a(y-t\omega,t)\,dt\Big)\,
{e^{-iy\cdot\xi}}\,dy\\
&=&\displaystyle\int_{\R}\int_{\R^{n}}a(x,t)\,e^{-ix\cdot\xi}{e^{-it\omega\cdot\xi}}\,dx\,dt\\
&=&\widehat{a}{(\xi,\omega\cdot\xi)}
=\widehat{a}(\xi,\tau),
\end{array}$$
where we have  set $(\xi,\tau)=(\xi,\omega\cdot\xi)\in E.$
 Bearing in mind that for any $t\in\R$, $\mbox{Supp}\,a(\cdot,t)\subset\Omega\subset B(0,r/2)$, we
 deduce that
 $$
 \int_{\R^{n}\cap B(0,\frac{r}{2}+T)}\mathcal{R}(a)(\omega,y)\,{e^{-iy\cdot\xi}}\,dy=\widehat{a}(\xi,\tau).
 $$
Then, in view of Lemma \ref{Lemma3.2}, we finish the proof of this lemma.
\end{proof}
\subsubsection{End of the proof of Theorem \ref{Theorem1}} We are now in position to complete the proof
of Theorem \ref{Theorem1}, using the result we have already obtained  and  an
analytic argument that is inspired by \cite{[R1]} and adapted for  our case .
For $\rho
>0$ and $\kappa\in (\mathbb{N}\cup\{0\})^{n+1}$, we put
 $$
 |\kappa|=\kappa_{1}+...+\kappa_{n+1},\,\,\,\,\,\,\,\,\,B(0,\rho)=\{x\in\R^{n+1},\,\,|x|<\rho\}.
 $$
 We state the following result which is proved in Appendix \ref{Appendix A} (see also \cite{[R32]}).
\begin{Lemm}\label{Lemma3.4}
Let $\mathcal{O}$ be a non empty open set of the unit ball $B(0,1)\subset
\R^{d}$, $d\geq2$, and let $F$ be an analytic function in $B(0,2),$ that
satisfy
$$\|\p^{\kappa}F\|_{L^{\infty}(B(0,2))}\leq \frac{M|\kappa|!}{(2\rho)^{|\kappa|}},\,\,\,\,\forall\,\kappa\in(\mathbb{N}\cup\{0\})^{d},$$
for some  $M>0$,  $\rho>0$ and $N=N(\rho)$. Then, we have
$$\|F\|_{L^{\infty}(B(0,1))}\leq N M^{1-\gamma}\|F\|_{L^{\infty}(\mathcal{O})}^{\gamma},$$
where $\gamma\in(0,1)$ depends on $d$, $\rho$ and $|\mathcal{O}|$.
\end{Lemm}
The above lemma claims conditional stability for the analytic continuation.
For classical results for this type, one can see  Lavrent'ev, Romanov and
Shishatski\v{\i} \cite{[R22]}.  For a fixed $\alpha>0$, we set
$F_{\alpha}(\tau,\xi)=\widehat {a}(\alpha(\xi,\tau))$, for all
 $(\xi,\tau)\in\R^{n+1}.$ It is easy to see that $F_{\alpha}$ is analytic and
we have
$$\begin{array}{lll} |\p^{\kappa} F_
{\alpha}(\xi,\tau)|=\left|\p^{\kappa}
\widehat{a}(\alpha(\xi,\tau))\right|&=&\left|\p^{\kappa}\displaystyle\int_{\R^{n+1}}a(x,t)\,{e^{-i\alpha
(t,x)\cdot(\xi,\tau)}}\,dx\,dt\right|\cr
&=&\left|\displaystyle\int_{\R^{n+1}}a(x,t)(-i)^{|\kappa|}\alpha^{|\kappa|}(x,t)^{\kappa}{e^{-i\alpha(x,t)\cdot(\xi,\tau)}}\,dx\,dt\right|.
\end{array}$$
This entails that
$$\begin{array}{lll}
|\p^{\kappa} F_ {\alpha}(\xi,\tau)|\leq \displaystyle\int_{\R^{n+1}}|a(x,t)| \alpha^{|\kappa|}(|x|^{2}+t^{2})^{\frac{|\kappa|}{2}}\,dx\,dt
\leq\|a\|_{L^{1}(Q_{r,*})}\,\,
\alpha^{|\kappa|}\,\,(2T^{2})^{\frac{|\kappa|}{2}}
\leq C \,\,\displaystyle\frac{|\kappa|!}{(T^{-1})^{|\kappa|}}\,\,e^{\alpha}.
\end{array}
$$
The, upon applying Lemma \ref{Lemma3.4} with $M=Ce^{\alpha}$, $2\rho=T^{-1},$
and $\mathcal{O}=\mathring{E}\cap B(0,1)$, where
$$
\mathring{E}=\{(\xi,\tau)\in\R\times\para{\R^{n}\setminus{\{0_{\R^{n}}\}}},\,\,\,|\tau|<|\xi|\},
$$
one  may find a constant $\mu\in(0,1)$ such that we have for all
$(\xi,\tau)\in B(0,1)$, the following estimation
$$
|F_{\alpha}(\xi,\tau)|=|\widehat{a}(\alpha(\xi,\tau))|\leq C e^{\alpha(1-\gamma)}\|F_{\alpha}\|_{L^{\infty}(\mathcal{O})}^{\gamma}.
$$
Now the idea is to find an estimate for the Fourier transform of $a$ in a
suitable ball. Using the fact that $\alpha\,\mathring{E}=\{\alpha(
\xi,\tau),\,(\xi,\tau)\in\mathring{E}\}=\mathring{E}$, we obtain for all
$(\xi,\tau)\in B(0,\alpha)$
\begin{eqnarray}\label{EQ3.31}
|\widehat{a}(\xi,\tau)|=|F_{\alpha}(\alpha^{-1}(\xi,\tau))|&\leq& C e^{\alpha(1-\gamma)}\,\|F_{\alpha}\|_{L^{\infty}(\mathcal{O})}^{\gamma}\cr
&\leq& C e^{\alpha(1-\gamma)} \|\,\widehat{a}\,\|^{\mu}_{L^{\infty}(B(0,\alpha)\cap\mathring{E})}\cr
&\leq&C e^{\alpha (1-\gamma)}\|\,\widehat{a}\,\|_{L^{\infty}(\mathring{E})}^{\gamma}.
\end{eqnarray}
The next step in the proof is to deduce an estimate that links the unknown
coefficient $a$ to the measurement
$\Lambda_{a_{2},b_{2}}-\Lambda_{a_{1},b_{1}}$. To obtain such estimate, we
need first to decompose the $H^{-1}(\R^{n+1})$ norm of $a$ into the following
way
$$\begin{array}{lll}
\|a\|_{H^{-1}(\R^{n+1})}^{2/\gamma}\!\!\!&=\!\!\!&\displaystyle\Big(\displaystyle\int_{|(\tau,\xi)|<\alpha}\!\!\!\!(1+|(\tau,\xi)|^{2})^{-1}
|\widehat{a}(\xi,\tau)|^{2}\, d\xi d\tau
+\!\!\displaystyle\int_{|(\xi,\tau)|\geq\alpha}\!\!\!\!\!(1+|(\tau,\xi)|^{2})^{-1}|\widehat{a}(\xi,\tau)|^{2}\, d\xi d\tau\,\Big)^{1/\gamma}\\
&\leq&  C\Big(\alpha^{n+1}\,\,\,\|\widehat{a}\|^{2}_{L^{\infty}(B(0,\alpha))}+\,\alpha^{-2}\,\,\|a\|^{2}_{L^{2}(\R^{n+1})}\Big)^{1/\gamma}.\\
\end{array}$$
Hence, in light of  (\ref{EQ3.31}) and Lemma \ref{Lemma3.3}, we get
$$\begin{array}{lll}
\|a\|^{2/\gamma}_{H^{-1}(\R^{n+1})}&\leq&\,C\displaystyle\Big(\alpha^{{n+1}}\,e^{2\alpha(1-\gamma)}\,(\lambda^{\delta}
\|\Lambda_{a_{2},b_{2}}-\Lambda_{a_{1},b_{1}}\|
+\frac{1}{\lambda^{\beta}})^{
2\gamma}+\alpha^{-2}\Big)^{1/\gamma}\\
&\leq&C\displaystyle\Big(\alpha^{\frac{n+1}{\gamma}} \,e^{\frac{2\alpha(1-\gamma)}{\gamma}}
\lambda^{2\beta}\|\Lambda_{a_{2},b_{2}}-\Lambda_{a_{1},b_{1}}\|^{2}
+\alpha^{\frac{n+1}{\gamma}}\,e^{\frac{2\alpha(1-\gamma)}{\gamma}}\,\lambda^{-2\beta}+\alpha^{-2/\gamma}\Big).
\end{array}$$
Let $\alpha_{0}>0$ be sufficiently large and assume that $\alpha>\alpha_{0}$.
Setting
$$
\lambda=\alpha^{\frac{n+3}{2\mu\gamma}}\,e^{\frac{\alpha(1-\mu)}{\mu\gamma}},
$$
and using the fact that  $\alpha>\alpha_{0},$ one can see that
$\lambda>\lambda_{0}$ and
$\alpha^{\frac{n+1}{\gamma}}\,e^{\frac{2\alpha(1-\gamma)}{\gamma}}\,\lambda^{-2\beta}=\alpha^{-2/\gamma}.$
This entails that
 $$\begin{array}{lll} \|a\|^{2/\gamma}_{H^{-1}(\R^{n+1})}&\leq&
C\Big(\alpha^{\frac{\beta(n+1)+\delta(n+3)}{\beta\gamma}}\,e^{\frac{2\alpha(\beta+\delta)(1-\gamma)}{\beta\gamma}}
\|\Lambda_{a_{2},b_{2}}-\Lambda_{a_{1},b_{1}}\|^{2}
+\alpha^{-2/\gamma}\Big)\\
&\leq&
C\,\displaystyle\Big(e^{N\alpha}\|\Lambda_{a_{2},b_{2}}-\Lambda_{a_{1},b_{1}}\|^{2}+\alpha^{-2/\gamma}\Big),
\end{array}$$
where $N$ depends on $\delta,\,\beta,\,n,$ and $\gamma$. The next step is to
minimize the right hand-side of the above inequality with respect to
$\alpha$. We need to take $\alpha$ sufficiently large. So, there exists a
constant $m>0$ such that if
$0<\|\Lambda_{a_{2},b_{2}}-\Lambda_{a_{1},b_{1}}\|<m$, and
 $$
 \alpha=\frac{1}{N}|\,\log\|\Lambda_{a_{2},b_{2}}-\Lambda_{a_{1},b_{1}}\|\,|,
 $$
then, we have the following estimation
\begin{eqnarray}\label{EQ3.32}
\|a\|_{H^{-1}(Q_{r,*})}\leq \|a\|_{H^{-1}(\R^{n+1})}&\leq&
C\Big(\|\Lambda_{a_{2},b_{2}}-\Lambda_{a_{1},b_{1}}\|+|\,\log\|\Lambda_{a_{2},b_{2}}-\Lambda_{a_{1},b_{1}}\|\,|^{-2/\gamma}\Big)^{\gamma/2}\cr
&\leq&C\Big(\|\Lambda_{a_{2},b_{2}}-\Lambda_{a_{1},b_{1}}\|^{\gamma/2}+|\log\|\Lambda_{a_{2},b_{2}}-\Lambda_{a_{1},b_{1}}\||^{-1}\Big).
\end{eqnarray}
Now if  $\|\Lambda_{a_{2},b_{2}}-\Lambda_{a_{1},b_{1}}\|\geq m$, we have
$$
\|a\|_{H^{-1}(Q_{r,*})}\leq C\|a\|_{L^{\infty}(Q_{r,*})}\leq \frac{2CMc^{{\gamma/2}}}{m^ {{\gamma/2}}}\leq
\frac{2CM}{m^{\gamma/2}}\|\Lambda_{a_{2},b_{2}}-\Lambda_{a_{1},b_{1}}\|^{\gamma/2},
$$
hence (\ref{EQ3.32})  holds. Let us now consider $\theta>1$ such that
$p:=s-1=\frac{n+1}{2}+2\theta.$ Use Sobolev's embedding theorem we find by
interpolating
$$\begin{array}{lll}
\|a\|_{L^{\infty}(Q_{r,*})}&\leq&C\|a\|_{H^{\frac{n}{2}+\theta}(Q_{r,*})}\\
&\leq& C \,\|a\|_{H^{-1}(Q_{r,*})}^{1-\eta}\,\|a\|_{H^{s-1}(Q_{r,*})}^{\eta}\\
&\leq&C\,\|a\|_{H^{-1}(Q_{r,*})}^{1-\eta},
\end{array}$$
for some $\eta\!\in\!(0,1)$. This completes the proof of Theorem
\ref{Theorem1}.
 This  will be a key ingredient in proving the result of the
next section.
\subsection{Stability for the potential}\label{subsection3.2}
This section is devoted to the proof of Theorem \ref{Theorem2}. By  means of
the geometrical optics solutions constructed in Section \ref{Sec2}, we will
show using the stability estimate we have already obtained for the absorbing
coefficient $a$, that the time dependent potential $b$ depends stably on the
Dirichlet-to-Neumann map $\Lambda_{a,b}$.  As before, given
$\omega\in\mathbb{S}^{n-1}$, $(a_{i},b_{i})\in\mathcal{A}(M_{1},M_{2})$
 such that $(a_{1},b_{1})=(a_{2},b_{2})$ in $\overline{Q}_{r}\setminus Q_{r,*}$, we set
$$
a=a_{2}-a_{1},\,\,\,\,b=b_{2}-b_{1}\,\,\,\,\,\mbox{and}\,\,\,\,\,\,\,A(x,t)
=(A^{-}A^{+})(x,t)=\exp\Big(-\displaystyle\frac{1}{2}\displaystyle\int_{0}^{t}a(x+(t-s)\omega,s)\,ds\Big),
$$
where $A^{-}$ and $A^{+}$ are given by
$$
A^{-}(x,t)=\exp\Big( \frac{1}{2}\int_{0}^{t}a_{1}(x+(t-s)\omega,s)\,ds \Big),\,\,\,\,A^{+}(x,t)
=\exp\Big(-\frac{1}{2}\int_{0}^{t}a_{2}(x+(t-s)\omega,s)\,ds\Big).
$$
In the rest of this section, we define $b$ in $\R^{n+1}$ by
 $b=b_{2}-b_{1}$ in $\overline{Q}_{r}$ and $b=0$ on $\R^{n+1}\setminus
\overline{Q}_{r}$. We start by giving a preliminary estimate that will be
used to prove the main statement of this section.
\begin{Lemm}\label{Lemma3.5}
Let $(a_{i},b_{i})\in \mathcal{A}(M_{1},M_{2})$, $i=1,\,2$. There exists
$C>0$ such that for any $\omega\in\mathbb{S}^{n-1}$ and $\varphi\in
\mathcal{C}_{0}^{\infty}(\mathcal{A}_{r})$, the following estimate holds
$$\begin{array}{lll}
\Big| \! \displaystyle\int_{0}^{T}\!\!\displaystyle\int_{\R^{n}} b(y-t\omega,t) \varphi^{2}(y)\,dy\,dt\Big|
\!\!\!\!&\leq&\!\!\!\!
C\Big(\lambda^{3}\|\Lambda_{a_{2},b_{2}}-\Lambda_{a_{1},b_{1}}\|+\lambda\|a\|_{L^{\infty}(Q_{r,*})}+\displaystyle\frac{1}{\lambda}  \Big)
\|\varphi\|^{2}_{H^{3}(\R^{n})},
\end{array}$$
for any $\lambda>0$ sufficiently large. Here $C$ depends only on $\Omega$,
$M_{1}$, $M_{2}$ and $T$.
\end{Lemm}
\begin{proof}{}
We start with the identity (\ref{EQ3.22}), except this time we will isolate
the electric potential
$$
\int_{0}^{T}\!\!\!\int_{\Omega}b(x,t)u^{+}u^{-}\,dx\,dt=\int_{0}^{T}\!\!\!\int_{\Gamma} (\Lambda_{a_{2},b_{2}}-\Lambda_{a_{1},b_{1}})(f_{\lambda})
u^{-}\,d\sigma\,dt-\int_{0}^{T}\!\!\!\int_{\Omega}a(x,t)\p_{t}u^{+} u^{-}\,dx\,dt.
$$
By replacing $u^{+}$ and $u^{-}$ by their expressions we get
\begin{eqnarray}\label{EQ3.33}
&&\displaystyle\int_{0}^{T}\!\!\!\int_{\Omega}b(x,t)\varphi^{2}(x+t\omega)A(x,t)\,dx\,dt=\displaystyle\int_{0}^{T}\!\!\!\int_{\Gamma}
(\Lambda_{a_{2},b_{2}}-\Lambda_{a_{1},b_{1}})(f_{\lambda})u^{-}\,d\sigma\,dt\cr
\,\,\,
&&-\displaystyle\int_{0}^{T}\!\!\!\int_{\Omega}b(x,t)\varphi(x+t\omega)A^{-}(x,t)e^{-i\lambda(x\cdot\omega+t)}r_{\lambda}^{+}(x,t)\,dx\,dt
-\displaystyle\int_{0}^{T}\!\!\!\int_{\Omega}b(x,t)r_{\lambda}^{+}(x,t)r_{\lambda}^{-}(x,t)\,dx\,dt\cr
&&-\displaystyle\int_{0}^{T}\!\!\!\int_{\Omega}a(x,t)\p_{t}u^{+}u^{-}\,dx\,dt-\int_{0}^{T}\!\!\!\int_{\Omega}b(x,t) \varphi(x+t\omega)
A^{+}(x,t)
e^{i\lambda(x\cdot\omega+t)}r^{-}(x,t)\,dx\,dt\cr
&&\,\,\,\,\,\,\,\,\,\,\,\,\,\,\,\,\,\,\,\,\,\,\,\,\,\,\,\,\,\,\,\,\,\,\,\,\,\,\,\,\,\,\,\,\,\,\,\,\,\,\,\,\,\,\,\,\,\,\,\,\,\,\,\,\,\,\,\,\,\,\,\,\,
\,\,\,\,\,\,\,\,\,\,\,\,\,\,\,\,\,\,\,\,\,
= \displaystyle\int_{0}^{T}\!\!\!\int_{\Gamma}(\Lambda_{a_{2},b_{2}}-\Lambda_{a_{1},b_{1}})(f_{\lambda})
u^{-}\,d\sigma\,dt+I^{'}_{\lambda}.
\end{eqnarray}
Then, in view of (\ref{EQ3.33}),  we have
$$\begin{array}{lll}
\displaystyle\int_{0}^{T}\!\!\!\int_{\Omega}\!b(x,t)
\varphi^{2}(x+t\omega)dxdt=\!\!
\displaystyle\int_{0}^{T}\!\!\!\int_{\Omega}\!b(x,t)\varphi^{2}(x+t\omega)(1\!-\!A)dxdt
\!+\!\!\displaystyle\int_{0}^{T}\!\!\!\int_{\Gamma}\!(\Lambda_{a_{2},b_{2}}\!-\!\Lambda_{a_{1},b_{1}})(f_{\lambda})u^{-}d\sigma
dt\!+\!I^{'}_{\lambda}.
\end{array}$$
From (\ref{EQ2.10}), (\ref{EQ2.18}) and using the fact that $a=a_{2}-a_{1}=0$
outside $Q_{r,*}$, we find
\begin{equation}\label{EQ3.34}
|I'_{\lambda}|\leq C\Big( \lambda\,\|a\|_{L^{\infty}(Q_{r,*})}+\frac{1}{\lambda} \Big)\|\varphi\|^{2}_{H^{3}(\R^{n})}.
\end{equation}
By the trace theorem, we get
 \begin{eqnarray}\label{EQ3.35}
\Big|\int_{0}^{T}\!\!\!\int_{\Gamma}(\Lambda_{a_{2},b_{2}}-\Lambda_{a_{1},b_{1}} )(f_{\lambda}) u^{-}\,d\sigma\,dt
\Big|&\leq&\|\Lambda_{a_{2},b_{2}}-\Lambda_{a_{1},b_{1}}\|
\|f_{\lambda}\|_{H^{1}(\Sigma)}
\|u^{-}\|_{L^{2}(\Sigma)}\cr
&\leq& C \lambda^{3}\|\Lambda_{a_{2},b_{2}}-\Lambda_{a_{1},b_{1}}\|\,\|\varphi\|^{2}_{H^{3}(\R^{n})}.
\end{eqnarray}
On the other hand, we have
\begin{equation}\label{EQ3.36}
\Big|  \displaystyle\int_{0}^{T}\!\!\!\int_{\Omega}\!b(x,t)\varphi^{2}(x+t\omega)(1\!-\!A)\,dx\,dt \Big|
\leq C \|a\|_{L^{\infty}(Q_{r,*})}\,\|\varphi\|^{2}_{H^{3}(\R^{n})}.
\end{equation}
Then, in light of (\ref{EQ3.34})-(\ref{EQ3.36}), taking to account that
$b=b_{2}-b_{1}=0$ outside $Q_{r,*}$ and using the change of variables
$y=x+t\omega$ we get
$$\Big| \int_{0}^{T}\!\!\!\int_{\R^{n}}b(y-t\omega,t)\varphi^{2}(y)\,dy\,dt\Big|\leq C \Big( \lambda^{3}\|\Lambda_{a_{2},b_{2}}
-\Lambda_{a_{1},b_{1}}\|+\lambda\|a\|_{L^{\infty}(Q_{r,*})}+\frac{1}{\lambda}\Big)\|\varphi\|_{H^{3}(\R^{n})}^{2}.$$
This completes the proof of the Lemma.
\end{proof}
Now the idea is to deduce an estimate for the light ray transform of the
time-dependent unknown coefficient $b$ in order to control thereafter  its
 Fourier transform.
\begin{Lemm}\label{Lemma3.6}
Let $(a_{i},b_{i})\in\mathcal{A}(M_{1},M_{2})$, $i=1,\,2$. There exists
$C>0$, $\delta>0$, $\beta>0$ and $\lambda_{0}>0$ such that for all
$\omega\in\mathbb{S}^{n-1}$, the following estimate holds
 $$
 \Big| \mathcal{R}(b)(y,\omega)\Big|\leq C
 \Big(\lambda^{\delta}\|\Lambda_{a_{2},b_{2}}-\Lambda_{a_{1},b_{1}}\|+\lambda^{\delta}\|a\|_{L^{\infty}(Q_{r,*})}
+\frac{1}{\lambda^{\beta}}\Big),\,\,\,\,\,\,\,\mbox{a.\,e}\,y\in\R^{n},
$$
for any $\lambda>\lambda_{0}$. Here $C$ depends only on $\Omega$, $T$,
$M_{1}$ and $M_{2}$.
\end{Lemm}
\begin{proof}{}
We proceed as in the proof of Lemma \ref{Lemma3.2}. We  consider  the
sequence  $(\varphi_{h})_{h}$ defined by (\ref{EQ3.26}) with
$y\in\mathcal{A}_{r}$. Since we have
$$\begin{array}{lll}
\Big|\displaystyle\int_{0}^{T}b(y-t\omega,t)\!\!\!&dt&\!\!\!\Big|=\Big|\displaystyle\int_{0}^{T}\int_{\R^{n}}b(y-t\omega,t)\,\varphi_{h}^{2}(x)\,dx\,dt
\Big|\\
&\leq& \Big| \displaystyle\int_{0}^{T}\!\!\!\int_{\R^{n}}\!\!b(x-t\omega,t)\varphi_{h}^{2}(x)dx\,dt\Big|
+\Big| \displaystyle\int_{0}^{T}\!\!\!\int_{\R^{n}}\!\!\!\Big(b(y-t\omega,t)-b(x-t\omega,t)\Big)\varphi_{h}^{2}(x)dx\,dt \Big|.
\end{array}$$
Then, by applying Lemma \ref{Lemma3.5} with $\varphi=\varphi_{h}$, and since
$|b(y-t\omega,t)-b(x-t\omega,t)|\leq C |y-x|$, we obtain
$$\begin{array}{lll}
\Big|\displaystyle\int_{0}^{T}b(y-t\omega,t)dt\Big| \leq\!
C\Big(\lambda^{3}\|\Lambda_{a_{2},b_{2}}-\Lambda_{a_{1},b_{1}}\|+\lambda\|a\|_{L^{\infty}(Q_{r,*})}+\displaystyle\frac{1}{\lambda}
\Big)
\|\varphi_{h}\|^{2}_{H^{3}(\R^{n})}+C\!\!\displaystyle\int_{\R^{n}}\!\!|x-y|\varphi_{h}^{2}(x)dx.
\end{array} $$
On the other hand, since $\|\varphi_{h}\|_{H^{3}(\R^{n})}\leq C h^{-3}$
 and  $ \displaystyle\int_{\R^{n}} |x-y|
\varphi_{h}^{2}(x)\,dx\leq \,C \,h$, we conclude that
$$
\Big|\int_{0}^{T}b(y-t\omega,t)\,dt\Big|\leq C \Big( \lambda^{3}\|\Lambda_{a_{2},b_{2}}-\Lambda_{a_{1},b_{1}}\|
+\lambda\|a\|_{L^{\infty}(Q_{r,*})} +\frac{1}{\lambda}\Big)h^{-6}+C\,h.
$$ Selecting $h$ small such that  $h=h^{-6}/\lambda$. Then, we find two
constants $\delta>0$ and $\beta>0$ such that
$$
\Big|\int_{0}^{T}b(y-t\omega,t) \Big|\leq C
\Big(\lambda^{\delta}\|\Lambda_{a_{2},b_{2}}-\Lambda_{a_{1},b_{1}}\|+\lambda^{\delta}\|a\|_{L^{\infty}(Q_{r,*})}
+\frac{1}{\lambda^{\beta}}    \Big).
$$
Using the fact that $b=b_{2}-b_{1}=0$
outside $Q_{r,*}$, we then conclude that for all
 $y\in\mathcal{A}_{r}$ and $\omega\in\mathbb{S}^{n-1}$,
 $$
 \Big| \int_{\R}b(y-t\omega,t)\,dt \Big|\leq  C
 \Big(\lambda^{\delta}\|\Lambda_{a_{2},b_{2}}-\Lambda_{a_{1},b_{1}}\|+\lambda^{\delta}\|a\|_{L^{\infty}(Q_{r,*})}
+\frac{1}{\lambda^{\beta}}    \Big).
$$
Next, by arguing  as in the
derivation of Lemma \ref{Lemma3.2}, we end up upper bounding the light-ray
transform of $b$,  for all $y\in\R^{n}$.
\end{proof}
At this point, it is convenient to recall that our goal is to obtain an
estimate for the Fourier transform of $b$ in a precise set. So, by proceeding
by a similar way as in the previous section, we get this result.
\begin{Lemm}\label{Lemma3.7}
Let $(a_{i},b_{i})\in\mathcal{A}(M_{1},M_{2})$, $i=1,\,2$. There exists
$C>0$, $\delta>0$, $\beta>0$ and $\lambda_{0}>0$, such that the following
estimate
$$|\widehat{b}(\xi,\tau)|\leq C \Big(\lambda^{\delta}\|\Lambda_{a_{2},b_{2}}-\Lambda_{a_{1},b_{1}}\|+\lambda^{\delta}\|a\|_{L^{\infty}(Q_{r,*})}
+\frac{1}{\lambda^{\beta}}\Big),\,\,\,\,\,\,\,\mbox{a.\,e}, (\xi,\tau)\in
E,$$ for any $\lambda>\lambda_{0}$. Here $C$ depends only on $\Omega$, $T$,
$M_{1}$ and $M_{2}$.
\end{Lemm}
Next, using the above estimation as well as the analytic continuation
argument, that is Lemma \ref{Lemma3.4}, we upper bound the Fourier transform
of $b$ in a suitable ball $B(0,\alpha)$  as follows
\begin{equation}\label{EQ3.37} |
\widehat{b}(\xi,\tau)|\leq Ce^{\alpha
(1-\gamma)}\Big(\lambda^{\delta}\|\Lambda_{a_{2},b_{2}}-\Lambda_{a_{1},b_{1}}\|
+\lambda^{\delta}\|a\|_{L^{\infty}(Q_{r,*})}+\frac{1}{\lambda^{\beta}}
\Big)^{\gamma},
\end{equation}
 for some $\gamma\in(0,1)$ and where $\alpha>0$ is
assumed to be sufficiently large.  Then, in order to deduce an estimate
linking the unknown coefficient $b$ to the measurement
$\Lambda_{a_{2},b_{2}}-\Lambda_{a_{1},b_{1}}, $ we control  the
$H^{-1}(\R^{n+1})$ norm of $b$ as follows
$$\begin{array}{lll}
\|b\|^{\frac{2}{\gamma}}_{H^{-1}(\R^{n+1})}\leq C\Big[\alpha^{n+1}\|\widehat{b}\|^{2}_{L^{\infty}(B(0,\alpha))}
+\alpha^{-2}\|b\|^{2}_{L^{2}(\R^{n+1})}   \Big]^{\frac{1}{\gamma}}.
\end{array}$$
So,  by the use of (\ref{EQ3.37}), we obtain the following inequality
\begin{equation}\label{EQ3.38}
\|b\|_{H^{-1}(\R^{n+1})}^{\frac{2}{\gamma}}\leq C \Big[
\alpha^{\frac{n+1}{\gamma}}e^{\frac{2\alpha(1-\gamma)}{\gamma}}\para{\lambda^{2\delta}\epsilon^{2}+\lambda^{2\delta}
\|a\|^{2}_{L^{\infty}(Q_{r,*})}+\lambda^{-2\beta}}+\alpha^{\frac{-2}{\gamma}}
\Big],
\end{equation}
where we have set $\epsilon=\|\Lambda_{a_{2},b_{2}}-\Lambda_{a_{1},b_{1}}\|$. In
light of Theorem \ref{Theorem1}, one gets
$$
\|b\|^{\frac{2}{\gamma}}_{H^{-1}(\R^{n+1})}\leq C \Big[ \alpha^{\frac{n+1}{\gamma}}e^{\frac{2\alpha(1-\gamma)}{\gamma}}\para{\lambda^{2\delta}\epsilon^{2}
+\lambda^{2\delta}\epsilon^{2\mu_{1}\mu_{2}}+\lambda^{2\delta}|\log \,\epsilon|^{-2\mu_{2}}+\lambda^{-2\beta}}+\alpha^{-\frac{2}{\gamma}}
 \Big],
$$
for some $\gamma,\,\mu_{1},\,\mu_{2}\in(0,1)$ and $\delta,\,\beta>0$. Let
$\alpha_{0}>0$ be sufficiently large and we take $\alpha>\alpha_{0}$. Setting
$$
 \lambda=\alpha^{\frac{n+3}{2\gamma\beta}} e^{\frac{\alpha(1-\gamma)}{\gamma\beta}}.
 $$
By $\alpha>\alpha_{0}$, we can assume $\lambda>\lambda_{0}$. Therefore, the
estimate (\ref{EQ3.38}) yields
$$
\|b\|^{\frac{2}{\gamma}}_{H^{-1}(\R^{n+1})}\leq C\Big[ e^{N\alpha}\para{\epsilon^{2}+\epsilon^{2s}+|\log\epsilon|^{-2\mu_{2}}}+\alpha^{-\frac{2}{\gamma}} \Big],
$$
for some $s,\,\mu_{1},\,\mu_{2}\in(0,1)$, and where $N$ is depending on
$n,\,\gamma,\delta$ and $\beta$. Thus, if $\epsilon$ is small, we have
\begin{equation}\label{EQ3.39}
\|b\|^{\frac{2}{\gamma}}_{H^{-1}(\R^{n+1})}\leq C\Big( e^{N\alpha}|\log
\epsilon|^{-2\mu_{2}}+\alpha^{-\frac{2}{\gamma}}  \Big).
\end{equation}
In order to minimize the
right hand side of the above inequality with respect to $\alpha$, we need to
take $\alpha$ sufficiently large. So, we select $\alpha$ as follows
$$\alpha=\frac{1}{N} \log|\log\epsilon|^{\mu_{2}},$$
where we have assumed that  $\epsilon <c\leq1$. Then, the estimate
(\ref{EQ3.39}) yields
$$\|b\|_{H^{-1}(Q_{r,*})}\leq \|b\|_{H^{-1}(\R^{n+1})}\leq C \Big( \log |\log \|\Lambda_{a_{2},b_{2}}-\Lambda_{a_{1},b_{1}}\||^{\mu_{2}}
\Big)^{-1}. $$
This completes the proof of Theorem \ref{Theorem2}.
\section{Determination of coefficients from boundary measurements and final data}\label{Sec4}
In this section, we prove  Theorem \ref{Theorem3} and  \ref{Theorem4}. We
will extend the stability estimates obtained in the first case to a larger
region $Q_{r,\sharp}\supset Q_{r,*}$.  We shall consider the  geometric
optics solutions constructed in Section \ref{Sec2}, associated with a
function $\varphi$ obeying supp $\varphi\cap \Omega=\emptyset.$ Note that
this time, we have more flexibility on the support of the function $\varphi$
and we don't need to assume that supp $\varphi\pm
T\omega\cap\Omega=\emptyset$ anymore. We recall  that the observations in
this case are given by the following  operator
$$\begin{array}{ccc}
\mathscr{R}_{a,b}:\mathcal{H}^{1}_{0}(\Sigma)&\longrightarrow&\mathcal{K}\\
\,\,\,\,\,\,\,\,\,\,\,f&\longmapsto&(\p_{\nu}u,u(\cdot,T),\p_{t}u(\cdot,T)),
\end{array}$$
associated to the problem (\ref{EQ1}) with $(u_{0},u_{1})=(0,0)$. We denote
by
$$
\mathscr{R}^{1}_{a,b}(f)=\p_{\nu}u,\,\,\,\,\,\mathscr{R}^{2}_{a,b}(f)=u(\cdot,T),\,\,\,\,\,\,\mathscr{R}_{a,b}^{3}(f)=\p_{t}u(\cdot,T).
$$
\subsection{Stability for the absorbing coefficient}\label{subsection4.1}
 In this section we will prove that the
absorbing coefficient $a$ can be stably recovered in a larger region if we
further know the final data of the solution $u$ of the dissipative wave
equation (\ref{EQ1}). In the rest of this section, we define $a=a_{2}-a_{1}$
in $\overline{Q}_{r}$ and $a=0$ on $\R^{n+1}\setminus \overline{Q}_{r}$. We
shall first prove the following statement
\begin{Lemm}\label{Lemma4.1}
Let $(a_{i},b_{i})\in\mathcal{A}(M_{1},M_{2})$, $i=1,\,2$. Let
$\varphi\in\mathcal{C}^{\infty}_{0}(\R^{n})$ be such that supp
$\varphi\cap\Omega=\emptyset$. There exists $C>0$, such that for any
$\omega\in\mathbb{S}^{n-1}$, the following estimate holds
$$\Big|\int_{\R^{n}}\varphi^{2}(y)\Big[ \exp\Big( -\frac{1}{2}\int_{0}^{T}a(y-s\omega,s)\,ds \Big)-1\Big] \,dy  \Big|\leq C\Big(\lambda^{2}
\|\mathscr{R}_{a_{2},b_{2}}-\mathscr{R}_{a_{1},b_{1}}\|+\frac{1}{\lambda}
\Big)\|\varphi\|_{H^{3}(\R^{n})}^{2}.$$ Here $C$ depends only on $\Omega$, $T$, $M_{1}$ and $M_{2}$.
\end{Lemm}
\begin{proof}{}
In view of Lemma \ref{Lemma2.1} and using the fact that supp
$\varphi\cap\Omega=\emptyset$, there exists a geometrical optic solution
$u^{+}$ to the wave equation
$$
\left\{
  \begin{array}{ll}
   \Big(\p_{t}^{2}-\Delta+a_{2}(x,t)\p_{t}+b_{2}(x,t)\Big)u^{+}=0  & \mbox{in}\,\,\,Q, \\
\\
   u^{+}(x,0)=\p_{t}u^{+}(x,0)=0  & \mbox{in}\,\,\,\Omega,
  \end{array}
\right.
$$
in the following form
\begin{equation}\label{EQ4.40}
u^{+}(x,t)=\varphi(x+t\omega)A^{+}(x,t)e^{i\lambda(x\cdot\omega+t)}+r_{\lambda}^{+}(x,t),
\end{equation}
corresponding to the coefficients $a_{2}$ and $b_{2}$,  where
$r_{\lambda}^{+}(x,t)$ satisfies (\ref{EQ2.9}) and (\ref{EQ2.10}). We denote
$$
f_{\lambda}(x,t)= u^{+}(x,t)_{|\Sigma}=\varphi(x+t\omega) A^{+}(x,t)e^{i\lambda(x\cdot\omega+t)}.
$$
Let $u_{1}$ be the solution of
$$\left\{
  \begin{array}{ll}
    \p_{t}^{2}u_{1}-\Delta u_{1}+a_{1}(x,t)\p_{t}u_{1}+b_{1}(x,t)u_{1}=0 & \mbox{in}\,\,\,Q, \\
\\
    u_{1}(x,0)=\p_{t}u_{1}(x,0)=0 & \mbox{in}\,\,\,\Omega, \\
\\
    u_{1}=f_{\lambda} & \mbox{on}\,\,\,\Sigma.
  \end{array}
\right.
$$
Putting $u=u_{1}-u^{+}$. Then, $u$ is a solution to the following system
\begin{equation}\label{EQ4.41}
\left\{
  \begin{array}{ll}
    \p_{t}^{2}u-\Delta u +a_{1}(x,t) \p_{t}u+b_{1}(x,t)u=a(x,t)\p_{t}u^{+}+b(x,t)u^{+} & \mbox{in}\,\,\,Q, \\
\\
    u(x,0)=\p_{t}u(x,0)=0 & \mbox{in}\,\,\,\Omega, \\
\\
    u(x,t)=0 & \mbox{on}\,\,\,\Sigma,
  \end{array}
\right.
\end{equation}
 where $a=a_{2}-a_{1}$ and $b=b_{2}-b_{1}$. On the other hand, Lemma \ref{Lemma2.3} guarantees the existence of a
geometrical optic solution $u^{-}$ to the adjoint problem
$$
\p_{t}^{2}u^{-}-\Delta u^{-} -a_{1}(x,t)\p_{t}u^{-}+(b_{1}(x,t)-\p_{t}a_{1}(x,t)) u^{-} =0 \,\,\,\,\,\,\,\, \mbox{in}\,\,\,Q,
$$
corresponding to the coefficients $a_{1}$ and $(-\p_{t} a_{1}+b_{1})$, in the
form
\begin{equation}\label{EQ4.42}
u^{-}(x,t)=\varphi(x+t\omega)e^{-i\lambda(x\cdot\omega+t)}A^{-}(x,t)+r_{\lambda}^{-}(x,t),
\end{equation}
 where $r_{\lambda}^{-}(x,t)$
satisfies (\ref{EQ2.17}) and (\ref{EQ2.18}). Multiplying the first equation
of (\ref{EQ4.41}) by $u^{-}$, integrating by parts and using Green's formula,
we get
\begin{eqnarray}\label{EQ4.43}
\displaystyle\int_{0}^{T}\!\!\!\int_{\Omega}a(x,t)\p_{t}u^{+}u^{-}dx dt\!\!\!\!&=&\!\!\!\!\!
\displaystyle\int_{0}^{T}\!\!\!\int_{\Gamma}(\mathscr{R}^{1}_{a_{2},b_{2}}-\mathscr{R}^{1}_{a_{1},b_{1}})(f_{\lambda})u^{-}(x,t)d\sigma dt
\!-\!\displaystyle\int_{\Omega}\!(\mathscr{R}^{3}_{a_{2},b_{2}}-\mathscr{R}^{3}_{a_{1},b_{1}})(f_{\lambda})u^{-}(x,T)\,dx\cr
&&\,\,\,\,\,-\displaystyle\int_{\Omega}(\mathscr{R}^{2}_{a_{2},b_{2}}-\mathscr{R}^{2}_{a_{1},b_{1}})(f_{\lambda})
\Big[a_{1}(x,T)u^{-}(x,T)-\p_{t}u^{-}(x,T)\Big]
\,dx\cr
&&\,\,\,\,\,\,\,\,\,\,\,\,\,\,\,\,\,\,\,\,-\displaystyle\int_{0}^{T}\!\!\!\int_{\Omega}b(x,t)u^{+}(x,t)u^{-}(x,t)\,dx\,dt.
\end{eqnarray}
By replacing $u^{+}$ and $u^{-}$ by their expressions,  using (\ref{EQ3.24})
and the Cauchy-Schwartz inequality, we obtain
$$\begin{array}{lll}
&&\!\!\!\!\!\!\!\!\!\!\!\!\Big|\displaystyle\int_{0}^{T}\!\!\!\int_{\Omega}a(x,t)\varphi^{2}(x+t\omega)A(x,t)\,dx\,dt
\Big|\leq\displaystyle\frac{C}{\lambda}
\Big[\Big(\|u^{-}\|^{2}_{L^{2}(\Sigma)}+\|u^{-}(\cdot,T)\|^{2}_{L^{2}(\Omega)}+\|\p_{t}u^{-}(\cdot,T)\|^{2}_{L^{2}(\Omega)}
\Big)^{\frac{1}{2}}\\
&&\!\!\!\!\!\!\!\Big(
\|(\mathscr{R}^{1}_{a_{2},b_{2}}-\mathscr{R}^{1}_{a_{1},b_{1}})(f_{\lambda})\|^{2}_{L^{2}(\Sigma)}
\!\!+\!\!\|(\mathscr{R}^{2}_{a_{2},b_{2}}-\mathscr{R}^{2}_{a_{1},b_{1}})(f_{\lambda})\|^{2}_{H^{1}(\Omega)}
\!\!+\!\!\|(\mathscr{R}^{3}_{a_{2},b_{2}}-\mathscr{R}^{3}_{a_{1},b_{1}})(f_{\lambda})\|_{L^{2}(\Omega)}^{2}\Big)^{\frac{1}{2}}\!\!+\!\!
\|\varphi\|^{2}_{H^{3}(\R^{n})}\Big].\\
\end{array}$$
Then, by setting
$\phi_{\lambda}=\displaystyle\para{u^{-}_{|\Sigma},\,u^{-}(\cdot,T),\,\p_{t}u^{-}(\cdot,T)}$,
one can see that
$$
\Big| \displaystyle\int_{0}^{T}\!\!\!\int_{\Omega}a(x,t) \varphi^{2}(x+t\omega)
A(x,t)\,dx\,dt\Big|\leq\frac{C}{\lambda}\Big(
\|\mathscr{R}_{a_{2},b_{2}}-\mathscr{R}_{a_{1},b_{1}}\|
\|f_{\lambda}\|_{H^{1}(\Sigma)}\|\phi_{\lambda}\|_{\mathcal{K}}+\|\varphi\|^{2}_{H^{3}(\R^{n})}\Big).
$$
Therefore, by the trace theorem we get
$$\Big| \int_{0}^{T}\!\!\!\int_{\Omega}a(x,t) \varphi^{2}(x+t\omega)
A(x,t)\,dx\,dt\Big|\leq C\Big(\lambda^{2}
\|\mathscr{R}_{a_{2},b_{2}}-\mathscr{R}_{a_{1},b_{1}}\|
+\frac{1}{\lambda}\Big)\|\varphi\|^{2}_{H^{3}(\R^{n})}.$$ Finally, we use the
fact that $a=a_{2}-a_{1}=0$ outside $Q_{r,\sharp}$ and we complete the proof
of the lemma by arguing as in the proof of Lemma \ref{Lemma3.1}.
\end{proof}
 Next, by considering the sequence $\varphi_{h}$ defined by (\ref{EQ3.26}) with $y\notin\Omega$,
 taking to account that  $a=a_{2}-a_{1}=0$ outside $Q_{r,\sharp}$
and  arguing as in Section \ref{subsection3.1}, we complete the proof of
Theorem \ref{Theorem3}.
\subsection{Stability for the  potential}\label{subsection4.2}
We are now in position to prove Theorem \ref{Theorem4}. We aim to show by the
use of Theorem\ref{Theorem3}, that the potential $b$ can be stably recovered
in the region $Q_{r,\sharp}$, with respect to the operator
$\mathscr{R}_{a,b}$ . In the rest of this section, we define $b$  in
$\R^{n+1}$ by $b=b_{2}-b_{1}$ in $\overline{Q}_{r}$ and $b=0$ on
$\R^{n+1}\setminus \overline{Q}_{r}$.
\begin{Lemm}\label{Lemma4.2}
Let $(a_{i},b_{i})\in \mathcal{A}(M_{1},M_{2})$, $i=1,\,2$. There exists
$C>0$ such that for any $\omega\in\mathbb{S}^{n-1}$ and $\varphi\in
\mathcal{C}_{0}^{\infty}(\R^{n})$ such that
$\mbox{supp}\,\varphi\cap\Omega=\emptyset$, the following estimate holds
$$
\Big|  \displaystyle\int_{0}^{T}\!\!\!\displaystyle\int_{\R^{n}} b(y-t\omega,t) \varphi^{2}(y)
\,dy\,dt\Big|
\leq
C\Big(\lambda^{3}\|\mathscr{R}_{a_{2},b_{2}}-\mathscr{R}_{a_{1},b_{1}}\|+\lambda\|a\|_{L^{\infty}(Q_{r,\sharp})}+\displaystyle\frac{1}{\lambda}
\Big)
\|\varphi\|^{2}_{H^{3}(\R^{n})},
 $$
where $C$ depends only on $\Omega$, $M_{1}$, $M_{2}$ and $T$.
\end{Lemm}
\begin{proof}{}
We start with the identity (\ref{EQ4.43}), except this time we  isolate the
potential $b$, we get
$$\begin{array}{lll}
\displaystyle\int_{0}^{T}\!\!\!\int_{\Omega}b(x,t)u^{+}u^{-}dx
dt\!\!\!\!&=&\!\!\!\!\!
\displaystyle\int_{0}^{T}\!\!\!\int_{\Gamma}(\mathscr{R}^{1}_{a_{2},b_{2}}-\mathscr{R}^{1}_{a_{1},b_{1}})(f_{\lambda})u^{-}(x,t)d\sigma
dt
\!-\!\displaystyle\int_{\Omega}\!(\mathscr{R}^{3}_{a_{2},b_{2}}-\mathscr{R}^{3}_{a_{1},b_{1}})(f_{\lambda})u^{-}(x,T)\,dx\cr
&&\,\,\,\,\,-\displaystyle\int_{\Omega}(\mathscr{R}^{2}_{a_{2},b_{2}}-\mathscr{R}^{2}_{a_{1},b_{1}})(f_{\lambda})\Big[a_{1}(x,T)u^{-}(x,T)
-\p_{t}u^{-}(x,T)\Big]
\,dx\cr
&&\,\,\,\,\,\,\,\,\,\,\,\,\,\,\,\,\,\,\,\,-\displaystyle\int_{0}^{T}\!\!\!\int_{\Omega}a(x,t)\p_{t}u^{+}(x,t)u^{-}(x,t)\,dx\,dt.
\end{array}$$
So, by replacing $u^{+}$ and $u^{-}$ by their expressions,  taking to account
(\ref{EQ3.34}), (\ref{EQ3.36}) and the fact that $a=a_{2}-a_{1}=0$ outside
$Q_{r,\sharp}$, and making the change of variables $y=x+t\omega$, we obtain
$$
\Big| \int_{0}^{T}\!\!\!\int_{\R^{n}}b(y-t\omega,t)\varphi^{2}(y)\,dy\,dt
 \Big|\leq C \Big( \lambda^{3}\|\mathscr{R}_{a_{2},b_{2}}-\mathscr{R}_{a_{1},b_{1}}\|
+\lambda\|a\|_{L^{\infty}(Q_{r,\sharp})}+\frac{1}{\lambda}
\Big)\|\varphi\|_{H^{3}(\R^{n})}^{2}.
$$
This completes the proof of the lemma.
\end{proof}
In order to complete the proof of Theorem \ref{Theorem4}, it will be enough
to consider the sequence $(\varphi_{h})$ defined by (\ref{EQ3.26}), with,\,
$y\notin\Omega$, use the fact $b=b_{2}-b_{1}=0$ outside $Q_{r,\sharp}$ and
repeat the same arguments of Section \ref{subsection3.2}
\section{Determination of coefficients from boundary measurements and final data by varying the initial data}\label{Sec5}
In the present section, we deal with the same inverse problem, except the set
of data, in this case, is made of the responses of the medium for all
possible initial data. For
$(a_{i},b_{i})\in\mathcal{C}^{2}(\overline{Q})\times\mathcal{C}^{1}(\overline{Q})$,
$i=1,\,2$, we define $(a,b)=(a_{2}-a_{1},b_{2}-b_{1})$ in $Q$ and
$(a,b)=(0,0)$ on $\R^{n+1}\setminus Q$. By proceeding  as in
the derivation of Theorem \ref{Theorem1} and Theorem \ref{Theorem3}, we prove
a $\log$-type stability estimate in the determination of the absorbing
coefficient  $a$ over the whole domain $Q$, from the knowledge of the
measurement $\mathcal{I}_{a,b}$.

To prove such estimate, we proceed as in Section \ref{subsection3.1} and
\ref{subsection4.1}, except  this time, we have more flexibility on the
support of the function $\varphi_{h}$ defined by (\ref{EQ3.26}). Namely, we don't
need to impose any condition on its support anymore (we fix $y\in\R^{n}$).

The same thing for the determination of the  time-dependent potential $b$. we
argue as in Section \ref{subsection3.2} and \ref{subsection4.2} to prove  a
$\log$-$\log$-type stability estimate in recovering the time dependent
coefficient $b$ with respect to the operator $\mathcal{I}_{a,b}$, over the
whole domain $Q$.
\appendix
\section{Proof of Lemma \ref{Lemma3.4}}\label{Appendix A}
In this section, we give the proof of Lemma \ref{Lemma3.4} for analytic
continuation. The proof is inspired from estimates given in
\cite{[R1]}[Theorem 3] for one variable analytic function. We simplify an
adapted proof for our case. In order to express the main goal of this section
we first prove the following Lemma
\begin{Lemm}\label{LemmaA.1}
Let $J$ be an open interval in $[-\frac{1}{5},\frac{1}{5}]$, and $g$ be an
holomorphic function in the unit disc $D(0,1)\subset \C$ satisfying
\begin{equation}\label{A.44}
|g(z)|\leq 1,\,\,\,\,\,|z|<1.
\end{equation}
Then, there exist $\gamma\in(0,1)$ and $N>0$ such that the following estimate
holds
$$
\|g\|_{L^{\infty}(B(0,\frac{1}{2}))}\leq N\|g\|_{L^{\infty}(J)}^{\gamma},
$$
 where  $N$ and $\gamma$ are depending only on
$|J|$.
\end{Lemm}
\begin{proof}{}
We should first notice that for all  $n\geq1$, there exist $(n+1)$ points
such that $$-\frac{1}{5}\leq x_{0}<...<x_{n}\leq \frac{1}{5},$$
 with $x_{i}\in \overline{J}$, $i=0,..,n$, and satisfying the following estimation
\begin{equation}\label{A.45}
x_{i}-x_{i-1}\geq \frac{|J|}{n+1},\,\,\,\,\mbox{for}\,\,i=1,...,n.
\end{equation}
Let $z\in \mathbb{C}$. We denote by
$$
P_{n}(z)=\sum_{i=0}^{n}g(x_{i}){\displaystyle\prod_{j\neq i}(z-x_{j})}{\displaystyle\prod_{j\neq i}(x_{i}-x_{j})^{-1}}.
$$
In order to prove this lemma, we need first to find an upper bound for
$|P_{n}(z)|$. To do that we first notice that for  $l'>l$ we have
$x_{l'}-x_{l}=\displaystyle\sum_{i=l+1}^{l'} (x_{i}-x_{i-1})$. Hence,
(\ref{A.45}) entails that
$$
\left\{
  \begin{array}{ll}
   (x_{j}-x_{i})\geq (j-i)\,\displaystyle\frac{|J|}{n+1}  & j>i, \\
   (x_{i}-x_{j})\geq (i-j)\, \displaystyle\frac{|J|}{n+1}  & j<i.
  \end{array}
\right.
$$
As a consequence we have the following estimation
\begin{eqnarray}\label{A.46}
\,\,\,\,\,\,\,\,\,
\prod_{j\neq i}|x_{i}-x_{j}|\geq \prod_{j=0}^{i-1}(i-j)
\frac{|J|}{(n+1)}\prod_{j=i+1}^{n}(j-i)\frac{|J|}{(n+1)} \geq i!
\frac{|J|^{i}}{(n+1)^{i}} (n-i)! \frac{|J|^{n-i}}{(n+1)^{n-i}}.
\end{eqnarray}
On the other hand, it is easy to see that for $|z|\leq\frac{1}{2}$ and
$x_{j}\in \overline{J}$, $j=0,..,n$, we have
$$
\displaystyle\prod_{j\neq i}|z-x_{j}|\leq \displaystyle\prod_{j\neq i} \para{|z|+|x_{j}|}\leq 1,
$$
Putting this together with (\ref{A.46}), we end up getting this result
\begin{equation}\label{A.47}
|P_{n}(z)|\leq \sum_{i=0}^{n}C_{n}^{i}\,\frac{(n+1)^{n}}{n!
|E|^{n}}\,\|g\|_{L^{\infty}(J)}\leq e\para{\frac{6}{|J|}}^{n}
\,\|g\|_{L^{\infty}(J)}.
\end{equation}
The next step of the proof is to control $|g(z)-P_{n}(z)|$. For this purpose,
let us introduce the following function: for all $\xi\in \mathbb{C}$, such
that $|\xi|=1$, we denote by
$$G(\xi)={g(\xi)}{(\xi-z)^{-1}\displaystyle\prod_{j=0}^{n}(\xi-x_{j})^{-1}}.$$
Applying the residue Theorem, one obtains the following identity
$$
\frac{1}{2i\pi}\int_{|\xi|=1}G(\xi)\,d\xi=\para{\mbox{Res}(G,z)+\sum_{k=0}^{n}\mbox{Res}(G,x_{k})}=\Big({g(z)-P_{n}(z)}\Big)
\displaystyle\prod_{j=0}^{n}(z-x_{j})^{-1}.
$$
From this and the hypothesis (\ref{A.44}), it follows that for $|z|\leq
\frac{1}{2}$, and $x_{i}\in\overline{J}$,
 we have
 \begin{equation}\label{A.48}
|g(z)-P_{n}(z)|\leq
2\,{\para{\frac{1}{2}+\frac{1}{5}}^{n+1}}\Big(1-\frac{1}{5}\Big)^{-(n+1)}\leq
2
\para{\frac{7}{8}}^{n}.
\end{equation}
Combining   (\ref{A.47}) with (\ref{A.48}), one gets
$$
||g||_{L^{\infty}(B(0,1/2))}\leq 2 \para{\frac{7}{8}}^{n}+e\,\para{\frac{6}{|J|}}^{n}\|g\|_{L^{\infty}(J)},\,\,\,\,\,n\geq
1.
$$
 To complete the proof of the lemma, we need to  minimize the right hand
side of the last estimate with respect to $n$. To this end, let us define the
following function
$$
\psi(x)=2 e^{-x\log(8/7) }+e
\,\|g\|_{L^{\infty}(J)}\,e^{x\,\log(6/|J|)},\,\,\,\,\,\,x\in\R.
$$
A simple calculation show that the function $\psi$ reaches a minimum at this
point
$$
x_{0}=\Big[\log \Big(\frac{48}{7 |J|}\Big)\Big]^{-1}\log\Big[ {\frac{\log (8/7)}{e\,\|g\|_{L^{\infty}(J)}\log (6/|J|) }}\Big].
$$
Then, we end up getting the desired result.
\end{proof}
We move now to establish the second  result by the use of Hadamard's
three-circle theorem and Lemma \ref{LemmaA.1}.
\begin{Lemm}\label{LemmaA.2}
Let $\varphi$ be an analytic function in $[-1,1]$, and I an open interval in
$[-1,1]$. We assume that there exist positive constants $M$ and $\rho$ such
that
\begin{equation}\label{A.49}
|\varphi^{(k)}(s)|\leq \frac{M k!}{ (2\rho)^{k}},\,\,\,\,\,\,\,\,\, k\geq 0,\,\,s\in [-1,1].
\end{equation}
Then, there exist $N=N(\rho,|I|)$ and $\gamma=\gamma(\rho,|I|)$ such that we
have
\begin{equation}\label {A.50}
 |\varphi(s)|\leq N
\|\varphi\|_{L^{\infty}(I)}^{\gamma}M^{1-\gamma},\,\,\,\,\,\mbox{for\,\,any}\,\,\,\,s\in[-1,1].
\end{equation}
\end{Lemm}
\begin{proof}{}
In light of (\ref{A.49}), we have for all $s\in[-1,1]$,
$$\begin{array}{lll}
\Big|\displaystyle\sum_{k\geq 0}\varphi^{(k)}(s)\frac{1}{k!} (z-s)^{k}\Big|&\leq& \displaystyle\sum_{k\geq 0} M (2\rho)^{-k}|z-s|^{k}.
\end{array}$$
This entails that for all $s\in[-1,1]$  and for all $z\in B(s,\rho)$, we have
the following estimation
\begin{eqnarray}\label{A.51}
 \Big|\displaystyle \sum _{k\geq
0}\varphi^{(k)}(s) \frac{1}{k!}(z-s)^{k}\Big|\leq M \displaystyle\sum _{k\geq 0}
(2\rho)^{-k} \rho^{k}\leq 2M,
\end{eqnarray}
which implies that $\varphi$ can be extended to an holomorphic function in
$D_{\rho}=\cup B(s,\rho)$ for $-1\leq s\leq 1$.  We need first  to construct
a specific open interval in $[-\frac{1}{5},\frac{1}{5}]$ to apply Lemma
\ref{LemmaA.1}. To this end, we notice that
\begin{equation}\label{A.52}
[-1,1]\subset\bigcup_{1\leq j\leq n_{0}}I_{j}=\bigcup_{1\leq j\leq n_{0}}\Big[s_{j}-\frac{\rho}{5},s_{j}+\frac{\rho}{5}\Big[,
\end{equation}
where we have putted  $s_{j}=-1+(2j-1) \rho/5$,
$5/\rho\leq n_{0}\leq 5/\rho+1/2$ and  assumed that
 $I_{j}\displaystyle\cap I_{j'}=\emptyset$, for all $j,\,j'=1,...n_{0},\,\,\,j\neq j'.$
Therefore, the open interval $I$ can be written as  the meeting of
$(I_{j}\cap I),$ for $1\leq j\leq n_{0}$ where
$$
(I_{j}\cap I)\displaystyle\bigcap_{j\neq j'} (I_{j'}\cap
I)=\emptyset,\,\,\mbox {for}\,\, j,\,j'=1,...,n_{0}.
$$
Now, we fix $j_{0}\in\{1,...,n_{0}\}$ such that $|I_{j_{0}}\cap
I|=\displaystyle\max_{1\leq j\leq n_{0}}|I_{j}\cap I|.$
We define  $J_{s_{j_{0}},\rho} =\frac{1}{\rho}(I_{j_{0}}\cap I-s_{j_{0}}).$
In light of (\ref{A.52}) , we deduce  that $J_{{s_{j_{0}}},\rho}$ is an open
interval of $[-\frac{1}{5},\frac{1}{5}]$. Next, we consider the function $g$
defined on $D(0,1)$ as follows
$$
g(z)=\frac{\varphi(s_{j_{0}}+\rho z)}{2M}.
$$
The estimate (\ref{A.51}) entails that $|g(z)|\leq 1$ for $|z|\leq 1$.
Bearing in mind that the function $g$ is holomorphic in the unit disc, we
deduce from Lemma \ref{LemmaA.1} the existence of two constants $N=N(|I|)$
and $\gamma=\gamma(|I|)$ such that the following estimate holds
$$
\|g\|_{L^{\infty}(B(0,1/2))}\leq N
\|g\|_{L^{\infty}(J_{s_{j_{0}},\rho})}^{\gamma}\leq
N(2M)^{-\gamma}\|\varphi\|_{L^{\infty}(I_{j_{0}}\cap I)}.
$$
 This combined with the fact that $\|g\|_{L^{\infty}(B(0,1/2))}=(2M)^{-1}\|\varphi\|_{L^{\infty}(B(s_{j_{0}},\rho/2))}$
yield the following result
\begin{equation}\label{A.54}
\|\varphi\|_{L^{\infty}(B({s_{j_{0}},\rho/2}))}\leq N \|\varphi\|^{\gamma}_{L^{\infty}(I)}M^{1-\gamma}.
\end{equation}
Now, we aim to extend this result to the interval $[-1,1]$. To this end, let
$r>0$, satisfying
\begin{equation}\label{A.55}
\frac{\rho}{2}\leq r\leq 2r\leq \rho.
\end{equation}
Let $(a_{j})_{j\geq 1}=(s_{j})_{j\geq 1}$ be a sequence such that
$
[-1,1]\subset \displaystyle\bigcup_{1\leq j\leq n_{0}}B(a_{j},2r)
$
and satisfying
\begin{equation}\label{A.56'}
\left\{
                              \begin{array}{ll}
                                B(a_{j+1},r)\subset B(a_{j},2r) & \mbox{for } \,\,j\in\{j_{0},...,n_{0}\} \\
\\
                                 B(a_{j-1},r)\subset B(a_{j},2r)& \mbox{ for}\,\,j\in\{1,...,j_{0}\}.
                              \end{array}
                            \right.
\end{equation}
In view of Hamdamard's three-circle theorem, using (\ref{A.54}) and (\ref{A.55}) we
get
\begin{eqnarray}\label{A.57'}
\|\varphi\|_{L^{\infty}(B(a_{j_{0}},2r))}\leq
\|\varphi\|_{L^{\infty}(B(a_{j_{0}},\frac{\rho}{2}))}^{\theta}\|\varphi\|_{L^{\infty}(B(a_{j_{0}},\rho))}^{1-\theta}
\leq N \|\varphi\|_{L^{\infty}(I)}^{\gamma}M^{1-\gamma},
\end{eqnarray}
where $\theta=\frac{\log {\rho/2r}}{\log 2}$. Then, using the fact that $B(a_{j+1},r)\subset B(a_{j},2r)$ for $j\in\{j_{0},...,n_{0}\}$, we deduce
$$\|\varphi\|_{L^{\infty}(B(a_{j_{0}+1},r))}\leq \|\varphi\|_{L^{\infty}(B(a_{j_{0}},2r))}\leq N\|\varphi\|_{L^{\infty}(I)}^{\gamma}M^{1-\gamma}.$$
From this and Hadamard's three-circle theorem, we obtain
$$\|\varphi\|_{L^{\infty}(B(a_{j_{0}+1},2r))}\leq \|\varphi\|_{L^{\infty}(B(a_{j_{0}+1},r))}^{\theta '}\|\varphi\|_{L^{\infty}(B(a_{j_{0}+1},\rho)}^{1-\theta'}\leq N\|\varphi\|_{L^{\infty}(I)}^{\gamma}M^{1-\gamma},$$
 where $\theta'=\frac{\log \rho/2r}{\log \rho/r}$. So, from (\ref{A.56'}) and a repeated application of Hadamard's three circle theorem, we get
$$\|\varphi\|_{L^{\infty}(B(a_{j},2r))}\leq N\|\varphi\|_{L^{\infty}(I)}^{\gamma}M^{1-\gamma},\,\,\,\,\,\,j\in \{j_{0}+2,...n_{0}\}.$$
By a similar way, we prove that
$$\|\varphi\|_{L^{\infty}(B(a_{j},2r))}\leq N\|\varphi\|_{L^{\infty}(I)}^{\gamma}M^{1-\gamma},\,\,\,\,\,\,j\in \{1,...j_{0}\}.$$
As a consequence, we obtain
$$
\begin{array}{lll}
\|\varphi\|_{L^{\infty}([-1,1]))}\leq \displaystyle\sum_{j=1}^{n_{0}}\|\varphi\|_{L^{\infty}(B(a_{j},2r))}\leq N\|\varphi\|^{\gamma}_{L^{\infty}(I)}M^{1-\gamma}.
\end{array}
$$
This completes the proof of the Lemma.
\end{proof}
\subsection{Proof of Lemma \ref{Lemma3.4}}
Notice first that there exists a sequence of open intervals $(I_{j})_{j}$
such that $$E=I_{1}\times...\times I_{j}\times...\times I_{d}\subset
\mathcal{O}\subset B(0,1).$$ Let $x=(x_{1},x_{2},...,x_{d})$ be fixed in
B(0,1). We consider the analytic function $\varphi_{j}$ defined as follows
\begin{equation}\label{A.56}
\varphi_{j}(s)=F(x_{1},...,x_{j-1},s,x_{j+1},...,x_{d}),\,\,\,\,\,\,\,s\in[-1,1].
\end{equation}
Assume  that there exist positive constants $M$ and $\rho$ such that
$$|\varphi_{j}(s)^{(k)}|\leq \frac{M k!}{(2\rho)^{k}},\,\,\,\,\,\,\,\,s\in[-1,1].$$
Then, in view of lemma \ref{LemmaA.2}, we conclude the existence of of
$N=N(\rho,|I_{j}|)$ and $\gamma=\gamma(\rho,|I|) $ such that we have
$$|\varphi_{j}(s)|\leq N\|\varphi_{j}\|_{L^{\infty}(I_{j})}^{\gamma_{j}}M^{1-\gamma_{j}},\,\,\,\,\,\,\,\,\,s\in[-1,1],$$
 This and (\ref{A.56}) yield
\begin{equation}\label{A.66}
|F(x)|\leq N_{j}\sup_{x_{j}\in
I_{j}}|F(x)|^{\gamma_{j}}M^{1-\gamma_{j}}.
\end{equation}
Therefore, by iterating (\ref{A.66}), we get
$$|F(x)|\leq N_{1}N_{2}^{\gamma_{1}}...N_{d}^{\gamma_{1}...\gamma_{d-1}}\sup_{x\in E}|F(x)|^{\gamma_{1}...\gamma_{d}}M^{1-\gamma_{1}...\gamma_{d}}.$$
This completes the of the lemma.



\begin{thebibliography}{99}
\bibitem{[R1]}{J. Apraiz, L. Escauriaza, }{\it Null-control and measurable
    sets, }{ESAIM: Control, Optimisation and calculus of variations, } 19,
239-254, (2013).
\bibitem{[R2]}{L. Baudouin, M. de Buhan, S. Ervedoza, }{\it Global Carleman
    estimates for waves and applications, }Communications in Partial
    Differential Equations, 38, 823-859, 2013.
\bibitem{[R3]}{L. Beilina, Nguyen Trung Th\`anh, M. V. Klibanov, J. B.
    Malmberg, }  {\it Reconstruction of shapes and refractive indices from
    backscattering experimental data using the adaptivity, } Inverse
    Problems, 30, 105007, 2014.
\bibitem{[R4]}{M. Bellassoued, }{\it Stable determination of coefficients in
    the dynamical Shr\"odinger equation in a magnetic field, }
arXiv:1510.04247v1.
\bibitem{[R5]}{M. Bellassoued, M. Choulli, M. Yamamoto, }{\it Stability
    estimate for an inverse wave equation and a multidimensional
    Borg-Levinson theorem, } J. Diff. Equat, 247, 2, 465-494, 2009.
\bibitem{[R6]}{M. Bellassoued, D. Dos Santos Ferreira, }{\it Stability
    estimates for the anisotripic wave equation from the Dirichlet-to-Neumann
    map.} Inverse Probl. Imaging, 5, 4, 745-73, 2011.
\bibitem{[R7]}{M. Bellassoued, D. Jellali, M. Yamamoto, }{\it Lipschitz
    stability for a hyperbolic inverse problem by finite local boundary data,
    } Applicable Analysis 85, 1219-1243, 2006.
\bibitem{[R8]}{M. Bellassoued, D. Jellali, M. Yamamoto, } {\it Stability
    estimate for the hyperbolic inverse boundary value problem by local
    Dirichlet-to-Neumann map, } J. Math. Anal. Appl. 343, 2, 1036-1046, 2008.
  \bibitem{[R9]}{M. Bellassoued, M. Yamamoto, }{\it Determination of a coefficient in the wave equation with a single measurement, }{App. Anal
      87, 901-920, 2008.}
\bibitem{[R10]}{I. Ben A\"icha Ibtissem, }{\it Stability estimate for a
    hyperbolic inverse problem with time-dependent coefficient, }{Inverse Problems 31 (2015) 125010 (21pp)}
\bibitem{[R11]}{R. Cipolatti, Ivo F. Lopez, }{\it Determination of
    coefficients for a dissipative wave equation via boundary measurements,}
    J. Math. Anal. Appl. 306, 317-329, 2005.
 \bibitem{[R12]}{M. Cristofol, S. Li, E. Soccorsi, }{\it Determining the waveguide conductivity in a hyperbolic equation from a single
     measurement on the lateral boundary,} arXiv: 1501, 01384, 2015.
     \bibitem{[R13]}{G. Eskin, }{\it A new approach to hyperbolic inverse
    problems, } Inverse problems, 22 no. 3, 815-831, 2006.
 \bibitem{[R15]}{M. Ikawa, }{\it Hyperbolic Partial
    Differential Equations   and Wave Phenomena}, Providence, RI American Mathematical
    Soc. 2000
  \bibitem{[R16]}{O. Imanuvilov, M. Yamamoto, } {\it Determination of a coefficient in an acoustic equation with single measurement, } Inverse
      Problems 19, 157-171, 2003.
     \bibitem{[R17]}{V. Isakov, }{\it An inverse hyperblic problem with many
    boundary measurements, } Commun. Partial Diff. Eqns., 16, 1183-1195, 1991.
\bibitem{[R18]}{V. Isakov, }{\it Completeness of products of solutions and
    some inverse problems for PDE, } J.Diff. Equat., 92, 305-316, 1991.
    \bibitem{[R19]}{V. Isakov, Z. Sun, }{\it Stability estimates for
    hyperbolic inverse problems with local boundary data,}  Inverse problems
    8, 193-206, 1992.
\bibitem{[R20]}{Y. Kian, }{\it Unique determination of a time-dependent
    potential for wave equations from partial data, } preprint,
    arXiv:1505.06498.
\bibitem{[R21]}{Y. Kian, }{\it Stability in the determination of a
    time-dependent coefficient for wave equations from partial data, }
    arXiv:1406.5734.
 \bibitem{[R22]}{M. M. Lavrent'ev, V. G. Romanov, S.P. Shishat$\cdot$ski\~{\i}, } {\it Ill-posed Problems of Mathematical Physics and Analysis, }
     Amer. Math. Soc., Providence, RI, 1986.
\bibitem{[R23]}{Rakesh, W. Symes,} {\it Uniqueness for an inverse
    problem for the wave equation,} Comm. in PDE, 13, 1, 87-96, 1988.

  \bibitem{[R24]}{A. G. Ramm, Rakesh,} {\it Property C and an Inverse Problem for a
    Hyperbolic Equation,} J.Math. Anal. Appl. 156, 209-219, 1991

\bibitem{[R25]}{A.G. Ramm, Sj\"{o}strand, }{\it An inverse inverse problem of
    the wave equation,} Math. Z., 206, 119-130, 1991.
\bibitem{[R26]}{R. Salazar, }{\it Determination of time-dependent
    coefficients for a hyperbolic inverse problem, } Inverse Problems, 29, 9,
    2013, 095015.
\bibitem{[R27]}{ P. Stefanov. } {\it  Uniqueness of the multi-dimentionnal
    inverse scattering problem for time-dependent potentials. }Math. Z., 201,
    4, 541-559, 1994.
\bibitem{[R28]}{P. Stefanov, G. Uhlmann, }{\it Recovery of a source term or a
    speed with one measurement and applications, } Trans. Amer. Math. Soc.,
    v.365, 5737--575, 2013.
\bibitem{[R29]}{P. Stefanov, G. Uhlmann, }{\it Stability estimates for the
    hyperbolic Dirichlet-to-Neumann map in anisotropic media, } J. Funct.
    Anal., 154,  330-358, 1998.
 \bibitem{[R30]}{Z. Sun, }{\it On continuous dependence for an inverse initial boundary value problem for the wave equation, } J. Math. Anal.
     Appl., 150, 188-204, 1990.
\bibitem{[R31]}{J. Sylvester, G. Uhlmann, }{\it A global uniqueness theorem
    for an inverse boundary value problem, }{Ann. of Math., 125 , 153-169,
    1987.}
    \bibitem{[R32]}{S. Vessella, }{\it A continuous dependence result in the analytic continuation problem, }Forum math, 11 no.6, 695-703, 1999.

\bibitem{[R33]}{A. Waters, }{\it Stable determination of X-ray transforms of
    time-dependent potentials from partial boundary data, }Commun. Partial
    Diff. Eqns., 39, 2169-2197, 2014.
\end{thebibliography}
\end{document}